\RequirePackage{lineno}
\documentclass[12pt,pra,aps,amssymb,amsfonts,amsmath,tightenlines, nofootinbib]{revtex4}
\usepackage{placeins}
\usepackage{dsfont}
\usepackage{graphicx}
\usepackage{caption}
\usepackage{amssymb,amsfonts,amsthm}
\usepackage{color}
\usepackage{verbatim}
\usepackage[normalem]{ulem}
\usepackage{enumerate}
\usepackage{mathrsfs}
\usepackage{bm,float}
\usepackage{mathtools}
\usepackage{extarrows}

\usepackage{textgreek}
\usepackage{lgreek}
\usepackage{gensymb}

\newtheorem{Proposition}{Proposition}

\usepackage{multirow}
\usepackage{rotating}
\usepackage{caption}
\usepackage{booktabs,siunitx,array,threeparttable}
\usepackage{tikz} 
\usetikzlibrary{shapes.geometric,calc,shapes,intersections,patterns,angles,quotes,through,backgrounds,decorations.markings,graphs}
\tikzset{myDot/.style={draw,shape=circle,fill=black,minimum size=6.0pt,inner sep=0.0pt,line width=1pt}} 
\tikzset{myDot2/.style={draw,shape=circle,fill=white,minimum size=6.0pt,inner sep=0.0pt,line width=1pt}}
\tikzset{edgeLabel/.style={font=\scriptsize, red!95!black,right}} 
\tikzset{extendLine/.style 2 args={shorten >=-{#2},shorten <=-{#1}}} 
\tikzset{myLine/.style={line width=.3mm}} 
\tikzset{myDottedLine/.style={dotted, line width=.3mm}} 
\tikzset{myDashedLine/.style={dashed, line width=.3mm}} 
\tikzset{myBoldLine/.style={line width=.6mm}} 

\begin{document}
\title{Tonnetz Theory, Classical Harmony, and the \\Combinatorial Geometry of Abstract Musical Resources}
\author{Jeffrey~R.~Boland$^1$ and Lane~P.~Hughston$^{2,3}$}

\affiliation{$^1$Syndikat LLC, 215 South Santa Fe Avenue, Los Angeles, California 90012, USA\\
$^{2}$School of Computing, Goldsmiths College, New Cross, London SE14\,6NW, UK\\
$^{3}$Artificial Intelligence and Mathematics Research Lab, James Carter Road, Mildenhall, Bury St Edmunds IP28\,7DE, UK}

\begin{abstract}
\noindent 
In a previous submission, we established a relation between tone networks and configurations.  It was shown that the Eulerian tonnetz can be represented by a $\{12_3\}$ of Daublebsky von Sterneck type D222.  We also constructed a tonnetz for Tristan-genus chords (dominant sevenths and half-diminished sevenths) and showed that this tonnetz can be represented by a $\{12_3\}$ of type D228. In both constructions the associated Levi graphs play an important role. Here we look at the tonnetze associated with some other musical systems, thereby offering concrete examples of an abstract view of music as combinatorial geometry. First, we look at the tonal harmonies of the classical period. In the case of diatonic triads, we show the existence of a bipartite graph of type $\{7_3\}$ and girth four that represents the relations between the seven diatonic degrees and their pitch classes. In the case of diatonic seventh chords, we obtain a Fano configuration $\{7_3\}$, which gives a characterization of the voice-leading relations that hold between such chords.
Next, we construct a tonnetz for pentatonic music based on the Desargues configuration $\{10_3\}$ and we construct a tonnetz for the twelve-tone system based on the Cremona-Richmond configuration $\{15_3\}$. Both can be used as resources for composition. Finally, we show that the relation between the chromatic pitch-class set and the major triad set is represented by a D222. The minor triads are in one-to-one correspondence with the members of a certain class of hexacycles in the Levi graph of this configuration. In this way, the characteristic duality between major and minor triads in the tonnetz can be broken. 
\begin{center}
\end{center}
\vspace{0.05cm}
\end{abstract}

\maketitle
%
\section{Introduction}
\label{sec:Introduction}

\noindent The Eulerian tonnetz is usually depicted in the form of a triangular tessellation of the Euclidean plane with pitch classes of the chromatic scale placed at its vertices. Another common representation takes the form of a tessellation of the plane by hexagons, with major and minor triads at the vertices. The two representations are dual in the sense that the faces of one correspond to the vertices of the other, and the edges are in one-to-one correspondence. These representations have the advantage of being planar but face the disadvantage of being infinite. If one gives up planarity then both versions of the tonnetz admit finite graphical representations. 
In Figure \ref{finite_Euler_tonnetz} we see a finite representation of the pitch-class tonnetz. In this regular graph of degree six, the twelve tones of the chromatic scale are arranged in a cycle of fifths. Each tone is joined to the six other tones with which it forms a consonant interval. Thus, the note $C$ is connected to $E_\flat$, $E$, $F$, $G$, $A_\flat$, and $A$. The resulting system of obtuse triangles, representing major and minor triads, comes equipped with an incidence relation, wherein the major triads act as ``points'' and the minor triads act as ``lines.''

\begin{figure}[htbp] 
\centering
\includegraphics[clip,scale=.65]{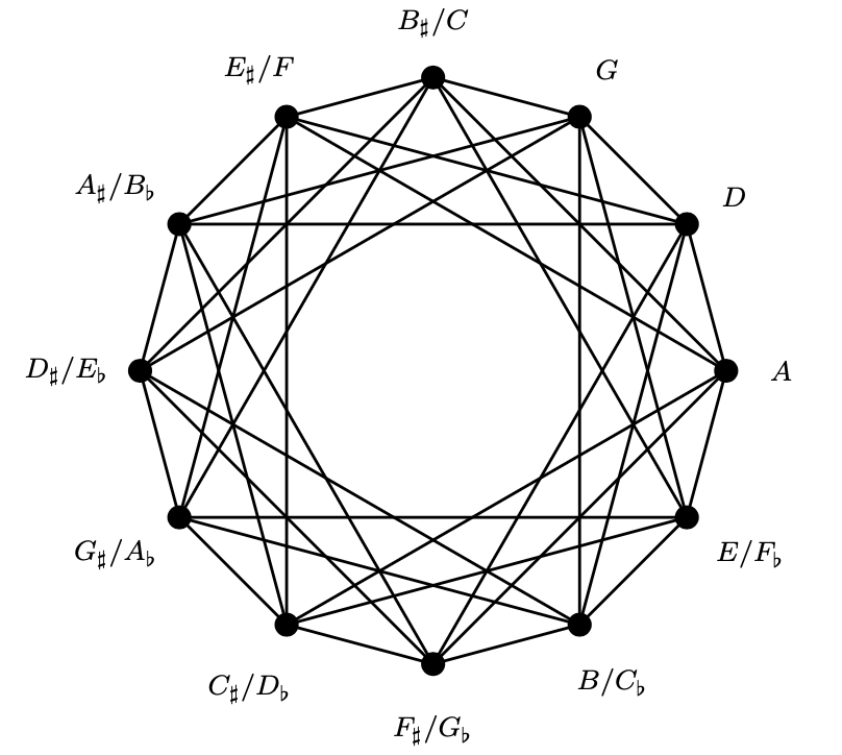}
\caption{The tonnetz, here depicted as a finite regular pitch-class graph of degree six, is dual to the Levi graph of the configuration $\{12_3\}$  of Daublebsky von Sterneck type D222. The notes of the chromatic scale are arranged in a cycle of fifths. The major and minor triads are represented by obtuse triangles such as $(C, E, G)$ and $(C, E_\flat, G)$. These 24 obtuse ``faces'' correspond to the 24 vertices of the dual Levi graph, shown in Figure \ref{Levi_graph_of_eulerian_tonnetz}. Each edge of the pitch-class graph belongs to one ``major'' triangle and one ``minor'' triangle.  Hence, there is a one-to-one correspondence between the edges of the pitch-class graph and the edges of the Levi graph. Thus, the edge $(C,E)$ belongs to the triangles $(C, E, G)$ and $(A, C, E)$, and corresponds to the edge $(CM, Am)$ in the ``fused triad'' Levi graph.  Each vertex in the pitch-class graph belongs to six obtuse triangles -- for example, the note $C$ belongs to $CM$, $FM$, $A_\flat$M, $Cm$, $Fm$, $Am$. These six triads form the unique $2p$-hexacycle in which each chord contains the note $C$. This cyclic ``face''  of the Levi graph corresponds to the vertex $C$ in the pitch-class tonnetz.}
\label{finite_Euler_tonnetz}
\end{figure}

\begin{figure}[htbp] 
\centering
\includegraphics[clip,scale=0.55]{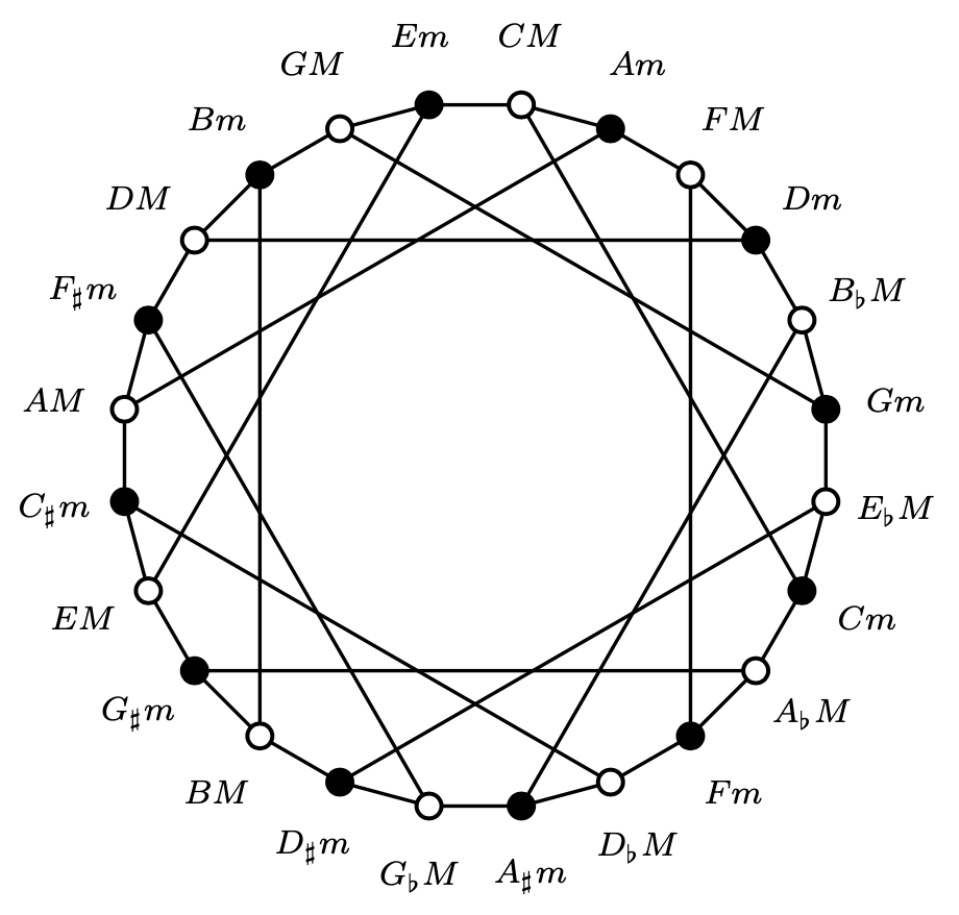}
\caption{The Levi graph of the tonnetz. The dual of the finite form of the Eulerian pitch-class tonnetz is a biregular graph of degree three and girth six, with twelve white vertices representing major triads and twelve black vertices representing minor triads. Each vertex of the Levi graph corresponds to an obtuse triangle of vertices in the pitch-class tonnetz.}
\label{Levi_graph_of_eulerian_tonnetz}
\end{figure}

Two triads are said to be incident if their triangles share an edge. In Figure \ref{Levi_graph_of_eulerian_tonnetz}, we see the dual of the pitch-class tonnetz, which takes the form of a Levi graph. By a Levi graph, we mean a finite bipartite graph that is biregular and of girth at least six. A bipartite graph is said to be {\em biregular} if all vertices of a given type have the same degree.  Each obtuse triangle of the pitch-class tonnetz represents  a vertex of the Levi graph, the major triads corresponding to white vertices and the minor triads to black vertices. Two vertices are adjacent on the Levi graph if and only if the associated triangles share an edge -- that is, if the two triads share a pair of tones. It is well known that each Levi graph determines a unique combinatorial configuration (Levi 1929, Coxeter 1950, Gr\"unbaum 2009). Hence, the tonnetz can be represented by such a configuration.  

In Boland and Hughston (2026) [henceforth, B-H (2026)],  it was shown that the configuration of the Eulerian tonnetz is the Daublebsky von Sterneck  $\{12_3\}$  of type D222. We also showed that a tonnetz for Tristan-genus chords (the twelve dominant seventh chords together with the twelve half-diminished sevenths) can be constructed based on the Daublebsky von Sterneck  $\{12_3\}$  of type D228. This remarkable tonnetz can be used for the study of the nineteenth-century harmonies typified in the later works of Richard Wagner. In Figure \ref{Figure RR} we show a fragment of the tessellation of the plane generated by the Levi graph of the D228, illustrating how a certain sequence of eight chords accompanying an important passage in Br\"unnhilde's final aria in G\"otterd\"ammerung maps to an octacycle in the Tristan-genus tonnetz. 

\begin{figure}[htbp] 
\centering
\includegraphics[clip,scale=0.57]{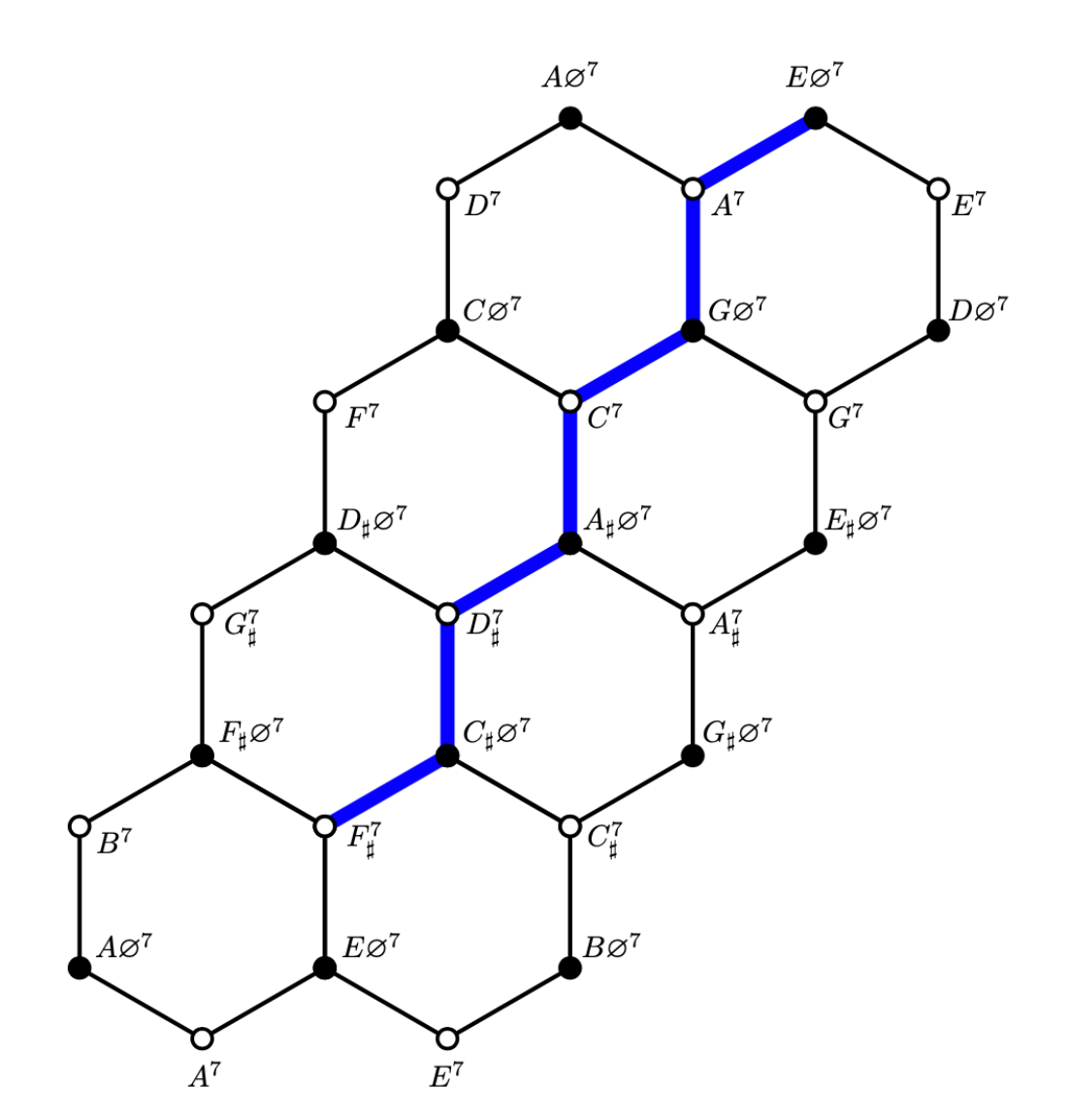}
\caption{Tessellation of the Tristan-genus tonnetz with the octacycle of the immolation progression from the final scene in Act III of G\"otterd\"ammerung marked in bold, starting at $E \varnothing ^7$.}
\label{Figure RR}
\end{figure}

Going forward, we assume the results and conventions of B-H (2026) and we defer to that work for general references concerning the Eulerian tonnetz and the tools we have used to establish relations between tonnetze and biregular graphs.  
The purpose of the present paper is to take this investigation further, with the construction of the following musical resources: (i) a set of tonnetze for the triads and seventh chords of diatonic tonal harmony, based on two distinct biregular graphs of type $\{7_3\}$, one of which is the Levi graph of the Fano configuration $\{7_3\}$; (ii) a tonnetz for pentatonic music based on the Desargues configuration $\{10_3\}$; and (iii) a tonnetz for twelve-tone music based on the Cremona-Richmond configuration $\{15_3\}$.  
We also undertake an analysis of the extent to which tonnetz constructions are predicated on voice-leading relations, or whether more general set-theoretic inclusion relations can be used as a basis for such constructions.

In Section \ref{sec:Tonal Harmony} we raise the question of how best to represent tonal harmony in the same spirit in which we have pursued the Levi graphs, configurations and tessellations associated with the Eulerian tonnetz. The literature of the tonnetz is curiously ambivalent on the matter of where it stands in relation to tonal harmony. Here we use the term ``tonal harmony''  in the sense of the treatments by Levarie (1954), Schoenberg (1978), Forte (1979), Piston (1987), and others, involving the major and minor scales along with the associated triads and seventh chords. We take it as given that tonal harmony deserves a mathematical treatment at least at the same level of perspicuity with which  investigations of the Eulerian tonnetz have been pursued. 

This we undertake and in doing so we put forward the premise that biregular graphs of type $\{7_3\}$ play a central role in the analysis of tonal harmony.  
An interesting feature of this analysis is the interplay of structures involving triads and tetrachords. 
We observe that triadic harmonies in major and minor modes can be interpreted in terms of the biregular graph of type $\{7_3\}$ given by Figure \ref{Figure AA}. This graph functions as a tonnetz for both modes and it admits a tessellation of the plane shown in Figure \ref{Figure DD},  intertwining three major triads, three minor triads, one diminished triad, and the pitch classes of which these chords are composed. The conclusions are summarized in Proposition \ref{Diatonic proposition}.

In the case of diatonic seventh chords an altogether  different type of tonnetz emerges, with the structure of a Fano $\{7_3\}$. We obtain a Levi graph and a configuration, the so-called Fano configuration, and we also obtain a tessellation of the plane -- the structures thus arising are shown in Figures \ref{Fano Configuration}, \ref{Figure Heawood} and \ref{Figure QQ} for the seventh chords of the key of $C$ major. The results are summarized in Proposition \ref{Seventh chord proposition}. Then we comment on the representation of chord progressions for triads and sevenths.

In Section \ref{Pentatonic Tone Network} we consider the idea that other systems of music might have their own tonnetze, and pursuant to this thought we construct such a resource for pentatonic musical structures.  We look at the formation of two-element and three-element subsets from a set of five elements. 

From this rather simple premise a surprisingly rich theory can be developed, leading to a tonnetz based on the Desargues configuration $\{10_3\}$. 
In Section \ref{Remarks on the Twelve-Tone System} an analogous construction is proposed for twelve-tone theory. We fix a hexachord -- that is to say, an unordered subset of six elements out of a set of twelve -- then we consider the fifteen ways of forming two-element subsets and the fifteen ways of forming triplets of disjoint two-element subsets. The resulting map has the structure of the Cremona-Richmond configuration $\{15_3\}$. 

The method of construction that we employed to assemble these configurations for pentatonic music and twelve-tone music is rather different from the method used in our analysis of the $\{12_3\}$ of the Eulerian tonnetz in B-H (2026), so one can ask if there is a more unified approach. 
The thinking behind many of the developments in tonnetz theory during the 1990s and 2000s was based on considerations of voice leading. The harmonic intuitions behind such thinking led to the  $\bf P$, $\bf L$ and $\bf R$ operations, and their generalizations, that tend to dominate these discussions even today and give explicit representations of the mechanics of semitonal voice leading. 

The central argument of Section \ref{sec:Subset relations vs voice-leading relations} is that we can put the consideration of voice leading to one side and develop the theory of the tonnetz from set-theoretic constructions. In many respects this is quite refreshing: it gives one tighter control of the mathematics and allows one to approach the theory in an objective way that avoids some of the prejudices or preconceptions that may be implicit in the theory of voice leading, such as those involving intuitive but elusive notions of distance, work, efficiency and parsimony. Our main results in this section are Propositions \ref{prop: incidence structures of note-major and note-minor} and \ref{prop: set inclusion relations}. We show that the tonnetz can be constructed on the basis of set-inclusion relations involving a pair of 12-sets: the pitch-class set and the major triad set. It is obvious that there is a natural three-to-three map between these sets. What is less obvious is that the resulting incidence structure is precisely that of the D222 of Daublebsky von Sterneck. A similar result holds for the minor triads, and the corresponding Levi graphs are drawn in Figure \ref{note-chord tonnetz}. When these are combined with the original Levi graph we obtain a tripartite tonnetz of major triads, minor triads and pitches, shown in  Figure \ref{fig:tripartite Levi graph}.

\section{Tonal Harmony}
\label{sec:Tonal Harmony}

\noindent Anyone new to the tonnetz will eventually be led to ask how the elaborate, finely developed common-practice-era edifice of tonal harmony ties in with the theory of the tonnetz. This is a natural question to ask, since at first glance the two systems seem so at odds. But what is the answer? Let us see what we can say about this. The principles of tonal harmony have been set out by many authors. The starting point is usually the choice of a major or minor key based on a set of seven pitch classes, which we label $\{1, 2, 3, 4, 5, 6, 7\}$. Then for $C$ major we have $\{C, D, E, F, G, A, B\}$ and for $C$ minor we have $\{C, D, E_{\flat}, F, G, A_{\flat}, B_{\flat}\}$. Associated with the seven pitches are seven triads, which we label with Roman numerals, I, II, III, IV, V, VI, VII, called scale degrees. 
\begin{figure}[htbp] 
\centering
\includegraphics[clip,scale=0.65]{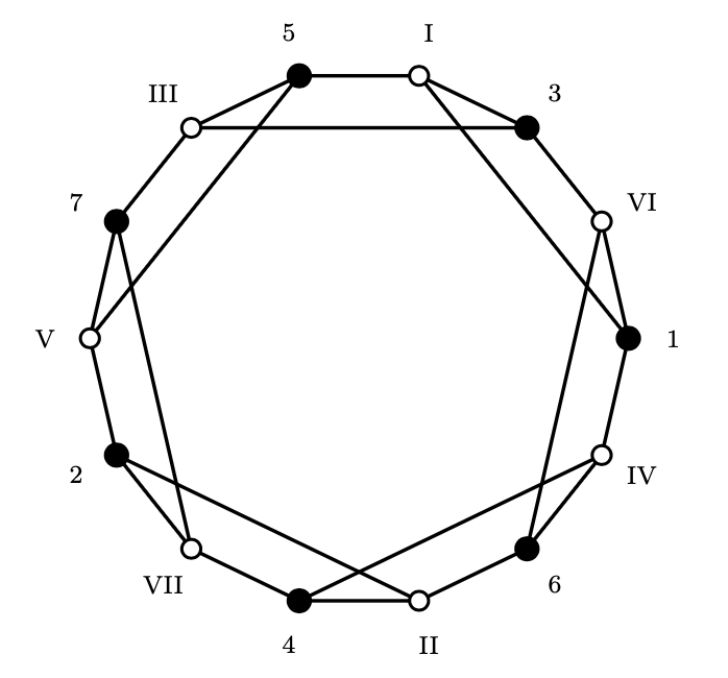}
\caption{Bipartite graph of type $\{7_3\}$ representing the pitch classes (labelled $1, 2, 3, 4, 5, 6, 7$) and triadic degrees (labelled $\rm I,\, II,\, III,\, IV,\, V,\, VI,\, VII$) associated with a diatonic key. Each such degree contains three pitch classes and each pitch class belongs to three triadic degrees.}
\label{Figure AA}
\end{figure}
In the case of $C$ major one sees that I $=\{C, E, G\}$, IV $=\{F, A, C\}$ and V $=\{G, B, D\}$ are major triads, that II $=\{D, F, A\}$, III $=\{E, G, B\}$ and VI $=\{A, C, E\}$ are minor triads, and that VII $=\{B, D, F\}$ is a diminished triad. 

Each  pitch class appears in precisely three of the seven triads. For example, in $C$ major the note $C$ appears in the triads $CM$, $FM$ and $Am$. The relations between the pitches and degrees can be summarized in a biregular graph of degree three in each vertex type with fourteen vertices, as in Figure \ref{Figure AA}. Each triad is joined by edges of the graph to the pitch classes that it includes and each pitch class is joined to the triads that include it.

\begin{figure}[htbp] 
\centering
\includegraphics[clip,scale=0.60]{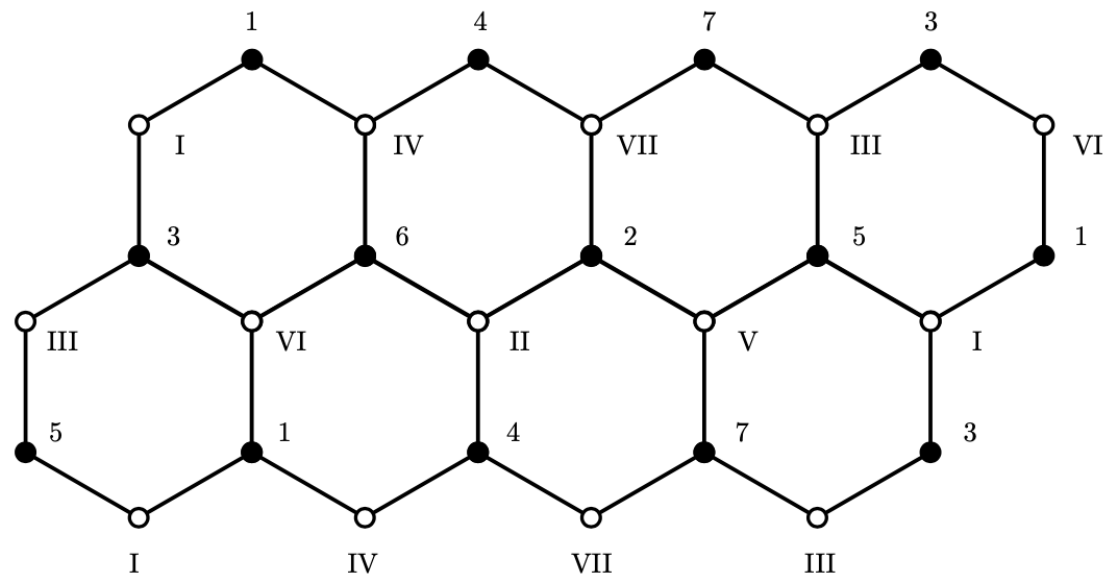}
\caption{The pitch classes and triadic degrees of a diatonic key determine a hexagonal tessellation of the Euclidean plane, of which a fragment is shown. Each degree is surrounded by the three pitches it contains and each pitch is surrounded by the three degrees that contain it.}
\label{Figure DD}
\end{figure}
If a bipartite graph is biregular, we say that the graph is of type  $\{ m_r, n_k \}$ for $m, n, r, k \in \mathbb N$ if there are $m$ white vertices and $n$ black vertices such that $k$ white vertices are connected to each black vertex and $r$ black vertices are connected to each white vertex. 
Then we write $\{ m_r \}$ for $\{ m_r, m_r \}$ and it follows that a graph of type $\{ m_r \}$ is the Levi graph of a combinatorial  configuration $\{ m_r \}$ if and only if its girth is no less than six. 
Figure \ref{Figure AA} depicts a graph of  type $\{ 7_3 \}$ and girth four. It admits tetracycles such as $\langle {\rm I}, 3, {\rm VI}, 1, {\rm I } \rangle$ and $\langle {\rm I}, 3, {\rm III}, 5, {\rm I} \rangle$ and hence is not a Levi graph. These are the only tetracycles containing the triad I. It is interesting to note that (a) in the major mode these tetracycles connect triad I to the minor triads linked to I respectively by a relative transformation and a leading-tone transformation, and (b) in the minor mode these tetracycles connect triad I to the two major triads linked to I by such transformations. We refer to Figure \ref{Figure AA} as {\it the triadic tonnetz of diatonic harmony} since it shows the relations between the seven pitches and the seven triads associated with a given key and mode. 

A corresponding tessellation of the plane can be constructed, shown in Figure \ref{Figure DD}. In both the tessellation  and the bipartite graph the significance of the tetracycles can be seen as follows. The triad I is linked to IV and to V in each case by a single black vertex, whereas III and VI are linked to I in each case by a pair of black vertices. The fact that one can get from III to I and from VI to I by more than one direct route leads to the existence of tetracycles. But there are no direct links between I and II or VII. 
This setup applies to both the major keys and minor keys of tonal harmony.  

Following Levarie (1954) and Schoenberg (1978), we use the natural minor as the basis for our analysis, regarding the harmonic and melodic minors as alterations used where required.
In the key of $C$ minor, one sees that I $=\{C, E_{\flat}, G\}$, IV $=\{F, A_{\flat}, C\}$ and V $=\{G, B_{\flat}, D\}$ are minor triads, that III $=\{E_{\flat}, G, B_{\flat}\}$, VI $=\{A_{\flat}, C, E_{\flat}\}$ and VII $=\{B_{\flat}, D, F\}$ are major triads, and that II $=\{D, F, A_{\flat}\}$ is a diminished triad. The graphs associated with $C$ major and $C$ minor are shown in Figure \ref{Figures BB and CC}. 
 
\begin{figure}[htbp] 
\centering
\includegraphics[clip,scale=0.60]{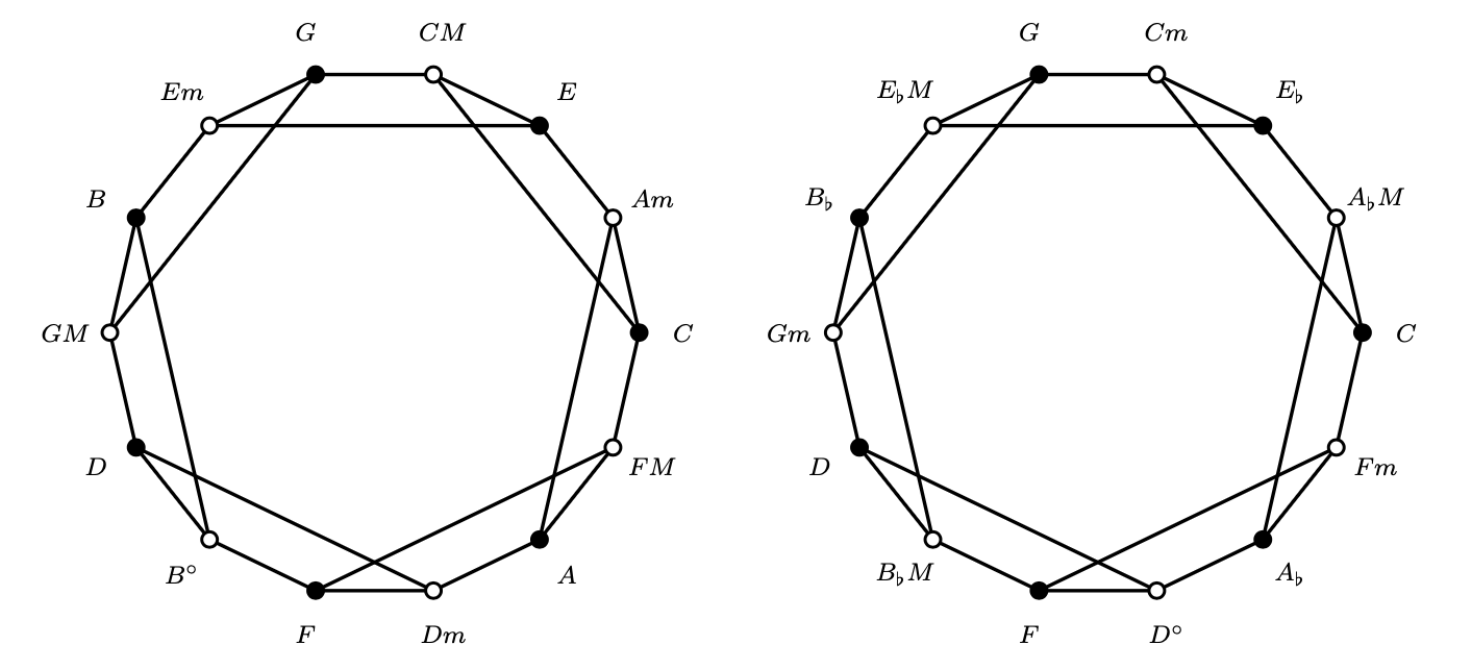}
\caption{In association with the key of $C$ major there are three major degrees (I, IV, V), three minor degrees (II, III, VI) and one diminished degree (VII). The seven triadic degrees together with the seven pitches form a bipartite graph of type $\{7_3\}$ with girth four. If two triads are adjacent to the same pitch, they share that pitch and hence a progression from one triad to the other can be made by pivoting on that pitch.  In the key of $C$ minor, there are three minor degrees (I, IV, V), three major degrees (III, VI,
VII) and one diminished degree (II).}
\label{Figures BB and CC}
\end{figure}

\begin{figure}[htbp] 
\centering
\includegraphics[clip,scale=0.60]{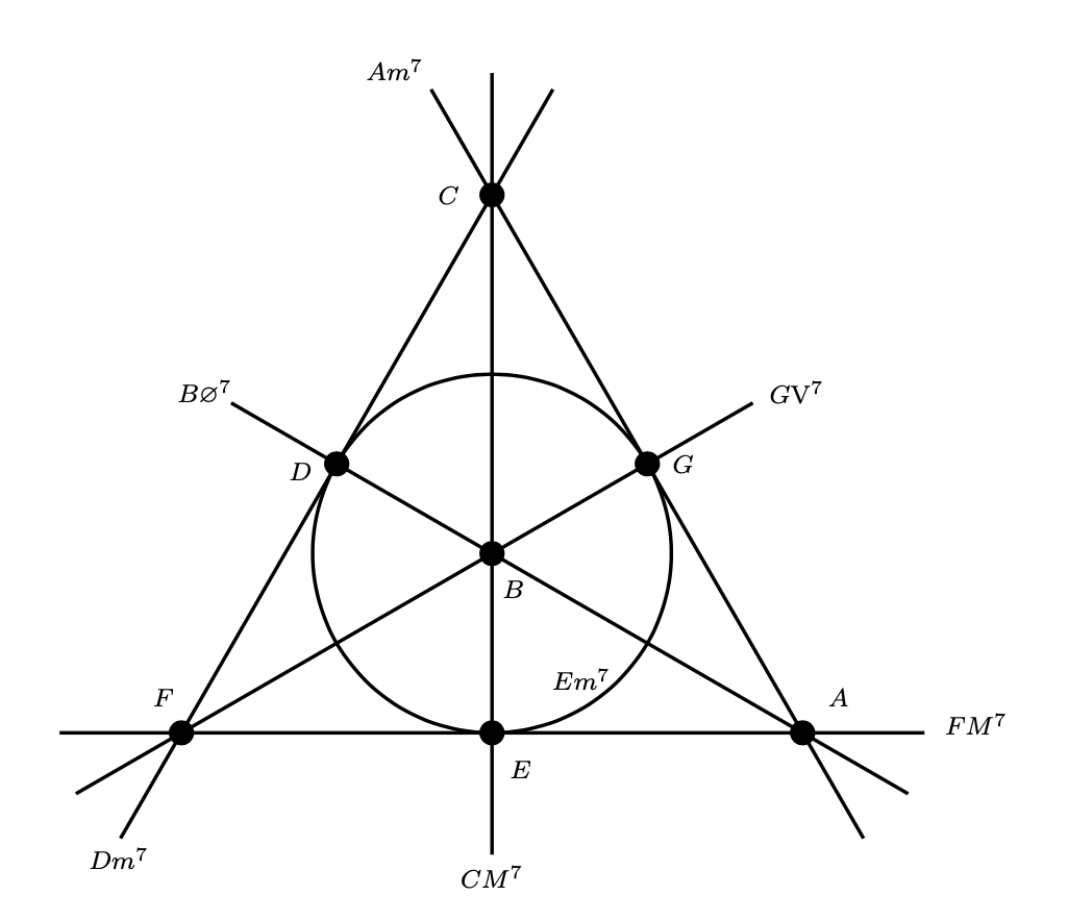}
\caption{A Fano configuration $\{7_3\}$ for the pitch classes and seventh chords of the key of $C$ major. The seven points represent tones and the seven lines, one demarcated by a circle, represent major sevenths, minor sevenths, dominant sevenths and half-diminished sevenths.}
\label{Fano Configuration}
\end{figure}

 \begin{Proposition}
A tonnetz can be constructed for diatonic triads in the form of a biregular graph of type $\{7_3\}$ and girth four, in accordance with which the seven diatonic pitch classes in a given mode are represented by black vertices and the seven associated diatonic triads are represented by white vertices. A corresponding hexagonal tessellation of the plane can be constructed in which the degrees of the triads ascend by fourths moving from left to right.  
\label{Diatonic proposition}
\end{Proposition}

Now let us turn to the matter of seventh chords. In the theory of tonal harmony, seventh chords are clearly important, but are often awarded the status of a kind of enriched triad, useful for adding variety and interest to triadic harmonies, and for strengthening progressions. 
But we have already seen in B-H (2026), in the context of pan-chromatic harmonies, that a system of seventh chords can form a tonnetz of its own, quite distinct from that of the Eulerian tonnetz for triads. Could something like that also be the case when seventh chords are considered in a diatonic context? 
 Let us investigate this point a little further. In tonal harmony, there are exactly seven seventh chords associated with each major key. For example, for the key of $C$ major, these are $\{CM^7, Dm^7, Em^7, FM^7, G{\rm V}^7, Am^7, B\varnothing^7\}$.  
 There are no other seventh chords that can be constructed from the notes of the $C$ major scale apart from the seven listed. 
 
 Here we use the notation $CM^7 = \{C, E, G, B\}$ for a major seventh, $G{\rm V}^7 = \{G, B, D, F\}$ for a dominant seventh, 
 $Dm^7 = \{D, F, A, C \}$ for a minor seventh, and $B\varnothing^7 = \{B, D, F, A \}$ for a half-diminished seventh. It is natural to see whether one can find a three-to-three map between the seven notes of a diatonic key and the seven-member set of seventh chords associated with it. 
 One way of doing this is to map each seventh chord to the three notes forming the underlying triad at the root of each chord. This merely gives back the triadic tonnetz of diatonic harmony that we have already discussed, and hence is uninteresting. 
 
 There is, however, another map that one can consider. This is to map each seventh chord to the tones of  the root, the third and the seventh of the chord. These three tones suffice to uniquely identify the seventh chord within the class of seventh chords associated with the given key. Then each tone of the diatonic key belongs to three such ``attenuated" seventh chords in this class and each such attenuated seventh chord contains three tones. This map gives rise to an interesting tonnetz that is quite distinct from the triadic tonnetz. The resulting structure is that of the Fano configuration $\{7_3\}$. The Fano configuration can be realized geometrically as the projective plane over the binary field $GF(2)$. 
In Figure \ref{Fano Configuration} we show the configuration of notes and chords for the diatonic sevenths of the $C$ major scale in the form of a Fano configuration and in Figure \ref{Figure Heawood} we show the Levi graph of this configuration, the so-called Heawood graph.  
 
 \begin{figure}[!htbp] 
\centering
\includegraphics[clip,scale=0.65]{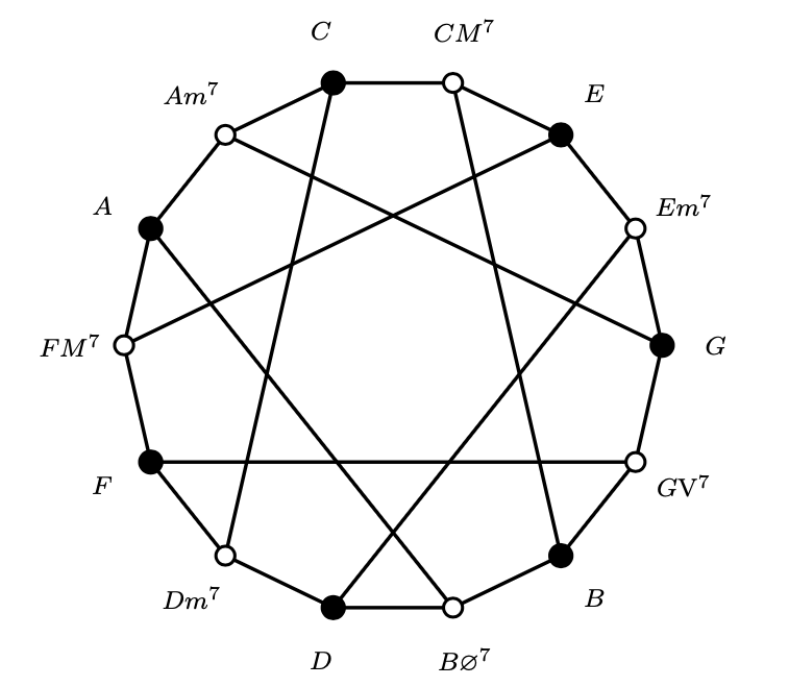}
\caption{The seven pitch classes and seven degrees of the seventh chords associated with the key of $C$ major can be arranged in the form of a Levi graph.  There are twenty-eight hexacycles altogether, falling into four classes, comprising seven ``beanies," seven ``keystones,'' seven ``bow-ties," and seven ``pilgrim hats.''}
\label{Figure Heawood}
\end{figure}

\begin{figure}[!htbp] 
\centering
\includegraphics[clip,scale=0.45]{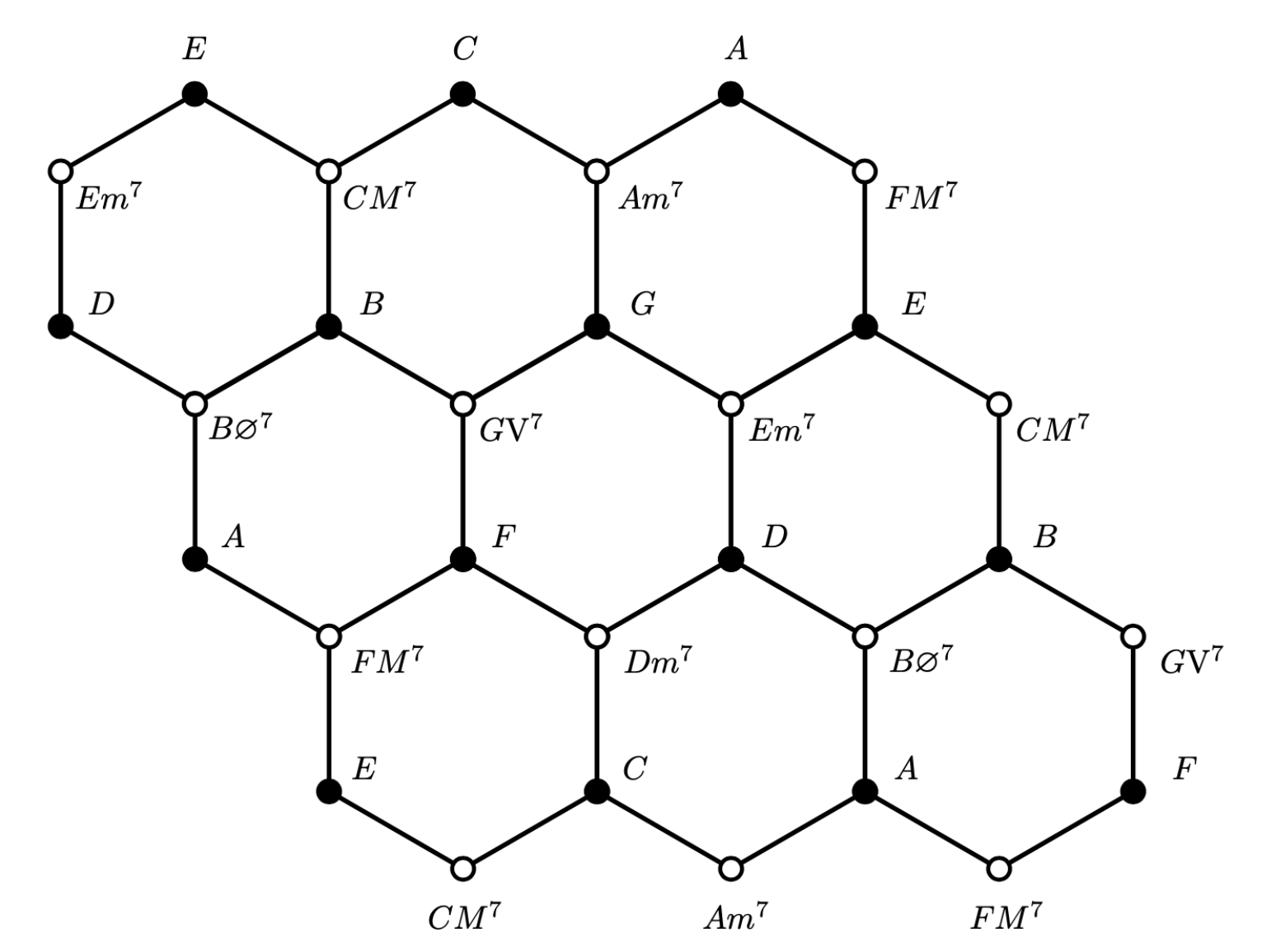}
\caption{Tessellation of the plane by the pitch classes and seventh chords of the $C$ major scale.}
\label{Figure QQ}
\end{figure}

In Figure \ref{Figure QQ} we show the corresponding tessellation.  It is interesting to observe that the seventh chords and the triads of the diatonic scale give rise to distinct tonnetze. This is reminiscent of the way in which the Eulerian tonnetz and the Tristan-genus tonnetz are distinct in the pan-triadic context and give rise to differing families of chord progressions. 

The situation with the seventh chords associated with the minor keys is similar. It will be useful to recall the relevant chords in the case of the key of $C$ minor, which are as follows: $\{Cm^7, D\varnothing^7, E_{\flat}M^7, Fm^7, Gm^7, A_{\flat}M^7, B_{\flat}\rm{V}^7\}$. Similarly, one constructs the Fano configuration, Levi graph, and tessellation for the seventh chords of the key of $C$ minor. 

\begin{Proposition}
A tonnetz can be constructed for diatonic seventh chords in the form of a bipartite graph of type $\{7_3\}$ with girth six that can be identified as the Levi graph of the Fano configuration. The seven diatonic pitch classes of the given mode are represented by black vertices and the associated seven degrees  of diatonic sevenths are represented by white vertices. A corresponding hexagonal tessellation of the plane can be constructed in which the degrees of the sevenths descend by thirds moving from left to right.  
\label{Seventh chord proposition}
\end{Proposition}

A remarkable feature of the tonnetz for the seventh chords of tonal harmony that is different from the situation with the triads is that the seventh chords are all linked to one another. This can be seen in the tessellation of Figure \ref{Figure QQ}, where one observes that each hexagon is surrounded by six other distinct hexagons. This ensures, for example, that $G{\rm V}^7$ is linked by at least one note to each of $CM^7$, $Dm^7$, $Em^7$, $FM^7$, $Am^7$, and $B\varnothing^7$. The fact that the seventh chords are linked in this way is not merely a curiosity. Rather, it is a manifestation of the fact that the Heawood graph can be embedded without crossings in a torus, which shows that to colour a map on a torus, at least seven distinct colours are necessary (Coxeter 1950). 

In the context of music theory, it is frequently remarked that one of the advantages of the use of seventh chords, apart from those already mentioned, is that one is offered more opportunities for smooth voice leading. One can strengthen the statement somewhat and say that one is guaranteed an opportunity for smooth voice leading from any one seventh chord to any other in the same key simply by holding over a common tone. 

In the case of a pair of triads belonging to a given diatonic scale, one finds that they have two notes in common, or one note, or none. But in the case of the sevenths associated with a given diatonic scale (major or minor) they  have three notes in common, or two, or one.  With this observation in mind, one can undertake a systematic study of progressions in the setting of the diatonic tonnetze. For instance, in a major mode both the progression from I to V in the form (I, VI, IV, II, V) and the return from V to I in the form (V, III, VI, II, V, I) can be charted on the tonnetz of Figure \ref{Figure DD}. See Forte (1979), examples 100, 101. One observes that the triadic tonnetz is indeed a natural medium for these model progressions. If one attempts the same exercise with the progressions from I to V and from V to I in a minor mode, given by (I, VII, III, VI, IV, V) and (V, VI, IV, V, I), respectively, then one checks that there is no minimal representation allowing a continuous movement through the tonnetz. This is because IV and V share no tones, nor do I and VII. But if one embeds these triads within seventh chords of the same degrees, then unique minimal trajectories for both the major and the minor progressions can be found on the Fano tonnetz.

\section{Pentatonic Tone Network}
\label{Pentatonic Tone Network}
\begin{quote} Musical keys and modulations do appear to arouse certain respondent moods or emotions. But is this, primarily, a matter of historical convention, of schooled expectation? What makes a minor third ``sad''? Is G minor, in the Western scale, intrinsically {\em triste} (and just what could such a statement signify?), or does its desolation stem from the use Mozart made of it in his great Quintet K\,516? What of the reflexes of sensibility motivated by keys, pitch, chordal blocks in non-Western tonic systems? Is a pentatonic structure any less universal than ours? 
\hfill ---George Steiner,~{\em Errata} (1997)
\end{quote}
\noindent  According to Piston (1987), ``The pentatonic scale, whose intervallic pattern corresponds to that of the black keys on the piano, is very ancient, having been used in the music of Oriental cultures perhaps longer even than the diatonic scale in the West. It is also the scale used in many European folksongs, particularly those of the British Isles.''  Examples of the use of pentatonic scales can be found in works of well-known modern Western composers such as Bart\'ok, Chopin, Debussy, Mahler, Ravel and Stravinsky. Could there be an analogue of the tonnetz based on a pentatonic scale?  Although rudiments of the system of major and minor triads can be glimpsed in the pentatonic system, nothing like the rich dodecaphonic chord system of Western music exists, so it is not easy to imagine offhand what form a pentatonic tonnetz might take. 

The theory of configurations suggests a way forward. If we have a system of music based on five tones, what kind of combinatorial geometry can be constructed? A hint comes from the work of Cayley (1846) who observed that five points in general position in $\mathbb {R}^3$ determine ten lines and ten planes that meet a generic plane in a configuration of ten points and ten lines such that three points lie on each line and three lines pass through each point (Coxeter 1950). Thus we obtain a $\{10_3\}$, which we can take as our pentatonic tone network. 

The resulting tonnetz has a very curious combinatorial structure. Instead of triads, the analogues of the major and minor chords are clusters of two and three notes, respectively. In pentatonic music these clusters possess a pleasing demeanour. If we take the pentatonic scale to be the notes 
$C,\, D,\, E, \, G, \, A$, the ten two-note clusters are $\{CD\}, \, \{CE\}, \, \{CG\}, \, \{CA\}, \, \{DE\}, \, \{DG\}, \, \{DA\}, \, \{EG\}, \, \{EA\}$, $\{GA\}$, and the ten three-note clusters are $\{ACD\}$,  $\{CDE\}$,  $ \{DEG\}$, $\{EGA\}$,  $\{GAC\}$,   $\{CDG\}$,  $\{DEA\}$,  $\{EGC\}$,  $\{GAD\}$,  $\{ACE\}$. 
The pleasurable effect of such three-note clusters is well known to Western musicians, even if they do not fit comfortably into the scheme of tonal harmony except as suspensions or as attenuated ninths or elevenths. 
For example, $\{ACD\}$ can be viewed as an $Am^{11}$, with some tones removed, and $\{CDE\}$ can be viewed as a thinned-out $CM^9$. 

An example can be found in Bach's Toccata and Fugue in F major for organ, at the entry of the fourth voice in the exposition of the fugue. It seems miraculous that the three upper voices converge to the tone cluster $\{CDE\}$ at the moment that the bass enters in the dominant on the note $C$ two octaves below. Bach makes it sound as if this near collision of the three upper lines is a happy accident. It is but a passing dissonance, achieved by a suspension of the $D$ from the previous measure,  in a series of such suspensions; and yet to those who know the piece it is a memorable moment, a calculated flash of intelligent deliberation on the part of the Maestro.
But the occurrence of such tone clusters via suspensions, even if frequent, tends to be sporadic in the common practice era and mostly for colour and effect rather than logic and syntax. Can one make music from clusters alone?

\begin{figure}[htbp] 
\centering
\includegraphics[clip,scale=0.7]{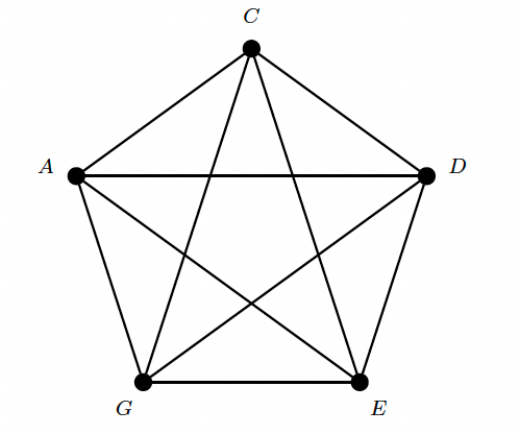}	
\caption{The complete pentagon, determined by five points $C,\, D,\, E, \, G, \, A$ and the ten lines pairwise joining them $CD, \, CE, \, CG, \, CA, \, DE, \, DG, \, DA, \, EG, \, EA, \, GA$, acts as an \textit{aide-m\'emoire} for the incidence relations of the ten lines and the ten triangles $ACD, \,CDE,  \,  DEG,  \, EGA, \,$ $ GAC, \, CDG, \, DEA, \, EGC, \, GAD, \, ACE$. The three lines forming the edges of a triangle are said to be incident with that triangle. Three triangles are then incident with each edge. For example, $CD$, $DE$, $CE$ are incident with $CDE$, whereas $CDA$, $CDG$, $CDE$ are incident with $CD$. } 
\label{pentagram}
\end{figure} 

Cayley's construction and the associated configuration of lines and planes can be easily remembered if one looks at the geometry of the pentacle shown in Figure \ref{pentagram} (cf.~Mason 1977, p.~164). The pentacle includes the five lines forming the pentagon together with five further lines forming the inscribed star pentagon, making a total of ten lines. Additionally, there are ten triangles -- five running around the edge of the pentagon, and five more forming the points of the star. 
On the one hand, each triangle consists of three lines, and on the other hand each line is incident with three distinct triangles. This gives a three-to-three map between lines and triangles, and hence defines an incidence relation between the set of lines and the set of triangles.

The theorem of Desargues is one of the high points of any treatment of  basic projective geometry -- see, e.g.,~that of Coxeter (1974). One of the curious features of this theorem is the fact that the analogous result in three dimensions is to some extent obvious, since the two distinct planes in which the perspectival triangles reside necessarily intersect in a line containing the three pairwise intersections of the edges, whereas in two dimensions the proof  is not so easily forthcoming. 

The resulting scheme of incidence involving the ten points and the ten lines gives rise to the Desargues configuration, which can be viewed both as geometric  and combinatorial. The fact that the Desargues configuration exists in the rational projective plane and that the relevant Diophantine problem has a solution hints at the deep connection this configuration admits with the theory of numbers. 

In fact, the musical labelling of the Desargues configuration shown in Figure \ref{fig:Desargues Configuration} is in some respects more memorable than the traditional geometric labelling. The five tones of the pentatonic scale yield ten two-note ``major'' clusters and ten three-note ``minor'' clusters. The  pitches of these tones may vary from culture to culture even if the combinatorial geometry of the configuration is invariant. One could even consider a form of pentatonic equal temperament by setting $C = 1$, \,$D = 2^{1/5}$, \,$E = 2^{2/5}$,  \,$G = 2^{3/5}$, \,$A = 2^{4/5}$. 
We refer to the resulting clusters of two and three tones as the ``chords'' of pentatonic music. Each major chord can be found in three minor chords and each minor chord contains three major chords.
For example, the major chord $CD$ can be found in the three minor chords $CDE$, $CDG$, and $CDA$, whereas the minor chord $CDE$ contains the major chords $CD$, $CE$, and $DE$, as one sees in Figure \ref{fig:Desargues Configuration}. Here, to ease the notation we drop the braces around unordered chords.
Thus we obtain a three-to-three map between the major chords and the minor chords. 

\begin{figure}[htbp]
\centering
\includegraphics[clip,scale=0.73]{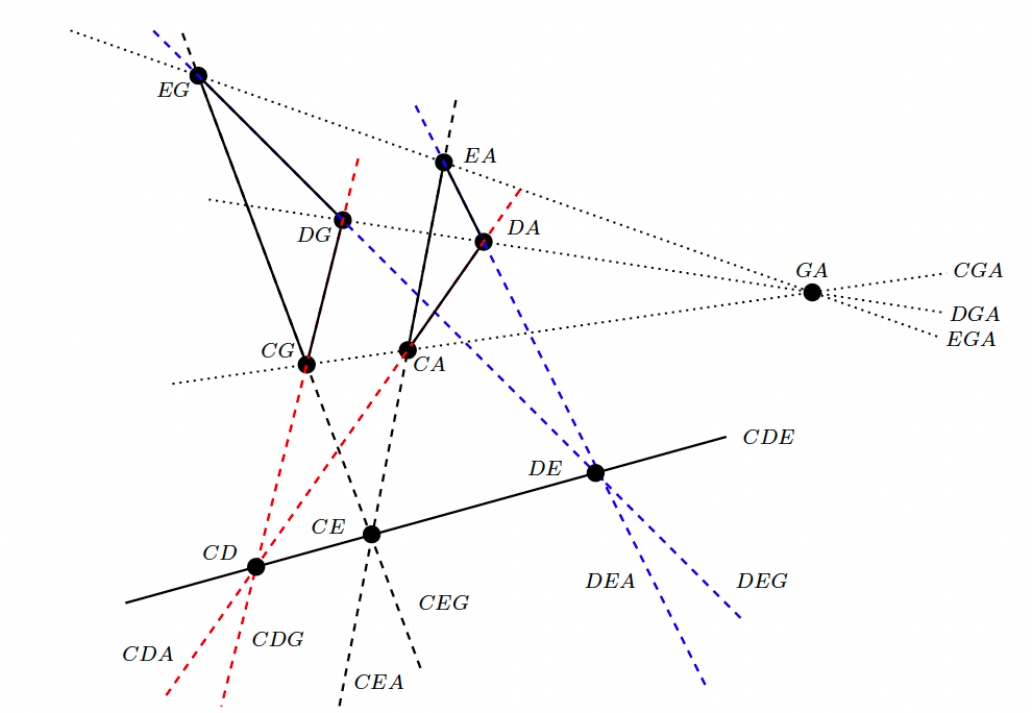}
\caption{The Desargues configuration as a pentatonic tonnetz. The triangles $(EG, DG, CG)$ and $(EA, DA, CA)$ are in perspective from the point $GA$. The lines $CEA$ and $CEG$ meet at the point $CE$, whereas $CDG$ and $CDA$ meet at $CD$, and the lines $DEA$ and $DEG$ meet at $DE$. The points $CD$, $CE$ and $DE$ are incident with the line of perspective $CDE$.}
\label{fig:Desargues Configuration}
\end{figure}

\begin{Proposition}
A tonnetz can be constructed for pentatonic music in the form of a self-dual $\{10_3\}$ in $\mathbb {R}^{2}$ known as the Desargues configuration, in accordance with which $10$ two-note ``major'' chords are represented by points and $10$ three-note ``minor'' chords are represented by lines. The $30$ incidence relations between the $10$ points and the $10$ lines determine the edges of the corresponding Levi graph.  
\label{Pentatonic proposition}
\end{Proposition}

Figure \ref{fig:DesarguesLevi} shows on the left how the chords of the pentatonic tone network can be arranged in the form of a Levi graph, of which another view can be found on the right in which some of the other cycles can be easily discerned (Coxeter 1950, Pisanski and Servatius 2013). What is remarkable is the overall affinity of the underlying combinatorial geometries of the pentatonic tonnetz and the Eulerian tonnetz. The results are summarized in Proposition \ref{Pentatonic proposition}. 

Looking at the Desargues configuration in Figure \ref{fig:Desargues Configuration} one might be tempted to assign special significance to the point of perspective marked by $GA$. But this is illusory -- any of the ten points can be taken as a point of perspective, from which the two associated triangles can be identified as well as the line of perspective. Once the symmetry is broken by the choice of a point of perspective the other points line up to it in a specific way. The choice of a chord, such as $GA$, as the point of perspective is analogous to a choice of key. No one choice is preferred {\em a priori}, but once a choice has been made, the music associated with that chord fixes relations between that chord and the other chords. 

\begin{figure}[htbp] 
\centering
\includegraphics[clip,scale=0.7]{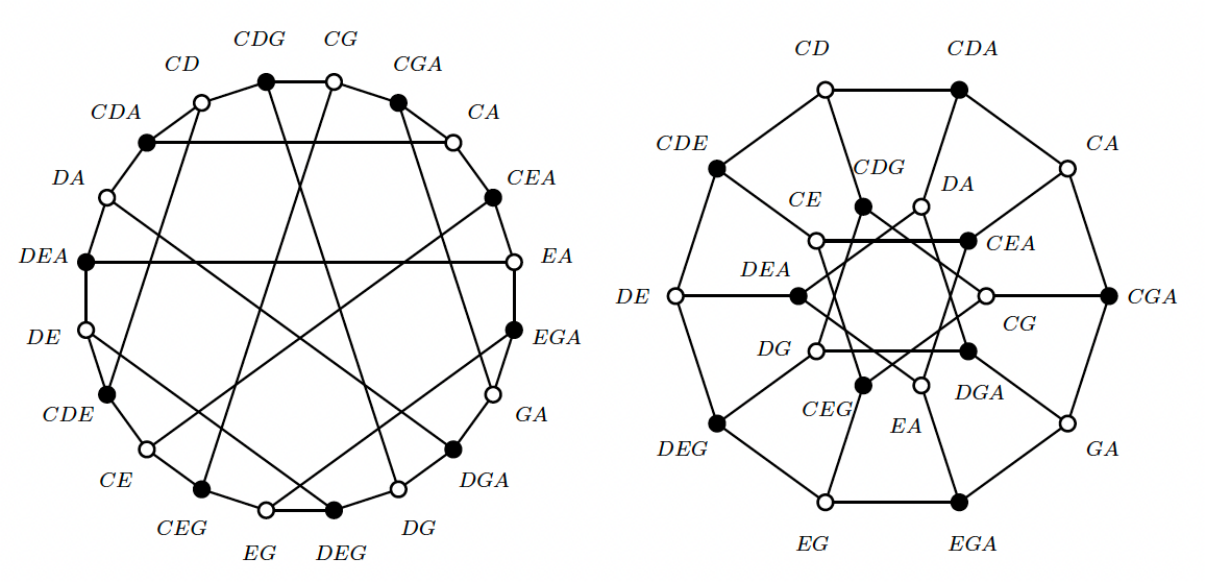}
\caption{The pentatonic tonnetz is represented here in two isomorphic ways as  the Levi graph of the Desargues configuration. Each two-note chord belongs to three distinct three-note chords and each three-note chord contains three distinct two-note chords. The girth of the graph is six, pointing to the existence of hexacycles, for example  $\langle CDA, \,CD, \,CDG, \,CG, \,CGA, \,CA, \,CDA \rangle$. The graph is manifestly self-dual. One also sees that there is a perimeter Hamiltonian cycle embracing all twenty major and minor chords. In fact, there are twenty-four such icosacycles. On the right, the twenty hexacycles of the pentatonic tonnetz are clearly visible.}
\label{fig:DesarguesLevi}
\end{figure}

For any system of music with an odd number of tones per octave, with or without equal temperament, the combinatorics of the pentatonic system can be extended as follows.

\vspace{0.20cm}

\begin{Proposition}
Let a scale contain an odd number $m = 2 k + 1$ of distinct tones, $k \in \mathbb N$. Then there are $m \choose k$ collections of $k$ tones, which we call $k$-chords,  that can be formed from the scale and there are $m \choose k+1$ = $ m \choose k$ collections of 
$k+1$ tones, which we call $(k+1)$-chords. Every $k$-chord is contained in $k+1$ \,$(k+1)$-chords and every $(k+1)$-chord contains $k+1$ \,$k$-chords. There exists a $k+1$ to $k+1$ map between the  $k$-chords and the $(k+1)$-chords and the associated biregular graph contains no tetracycles. Hence, if we write $a$ = $m \choose k$ and $b = k+1$, the construction yields a self-dual combinatorial configuration of type $\{a_b\}$. 
\label{Generalization proposition}
\end{Proposition}

\vspace{0.20cm}

Thus, we have a family of tonnetze based on sets of odd cardinality. If $m = 5$ we recover the pentatonic tonnetz. For a heptatonic scale ($m = 7$) one obtains a configuration $\{ 35_4 \}$ of 3-note and 4-note chords. For $m = 9$, we get a $\{ 126_5 \}$ of 4-note and 5-note chords. Even the case $k=1$ is interesting, which gives rise to a scale with three tones, together with a system of three one-tone major chords and three two-tone minor chords, for which the associated Levi graph consists of a single hexacycle.  

The pentatonic tonnetz can be used as a compositional tool. In Figure 
\ref{Fig:On_the_Perimeter} we give an example of a composition based on this tonnetz. On the left side in Figure \ref{fig:DesarguesLevi} one observes that there are five overlapping ``beanie"-style hexacycles running around the perimeter Hamiltonian cycle of the graph. The composition involves playing the chord sequence of each such hexacycle twice. 
 
 \begin{figure} [htbp]
 \centering
\includegraphics[scale=0.50]{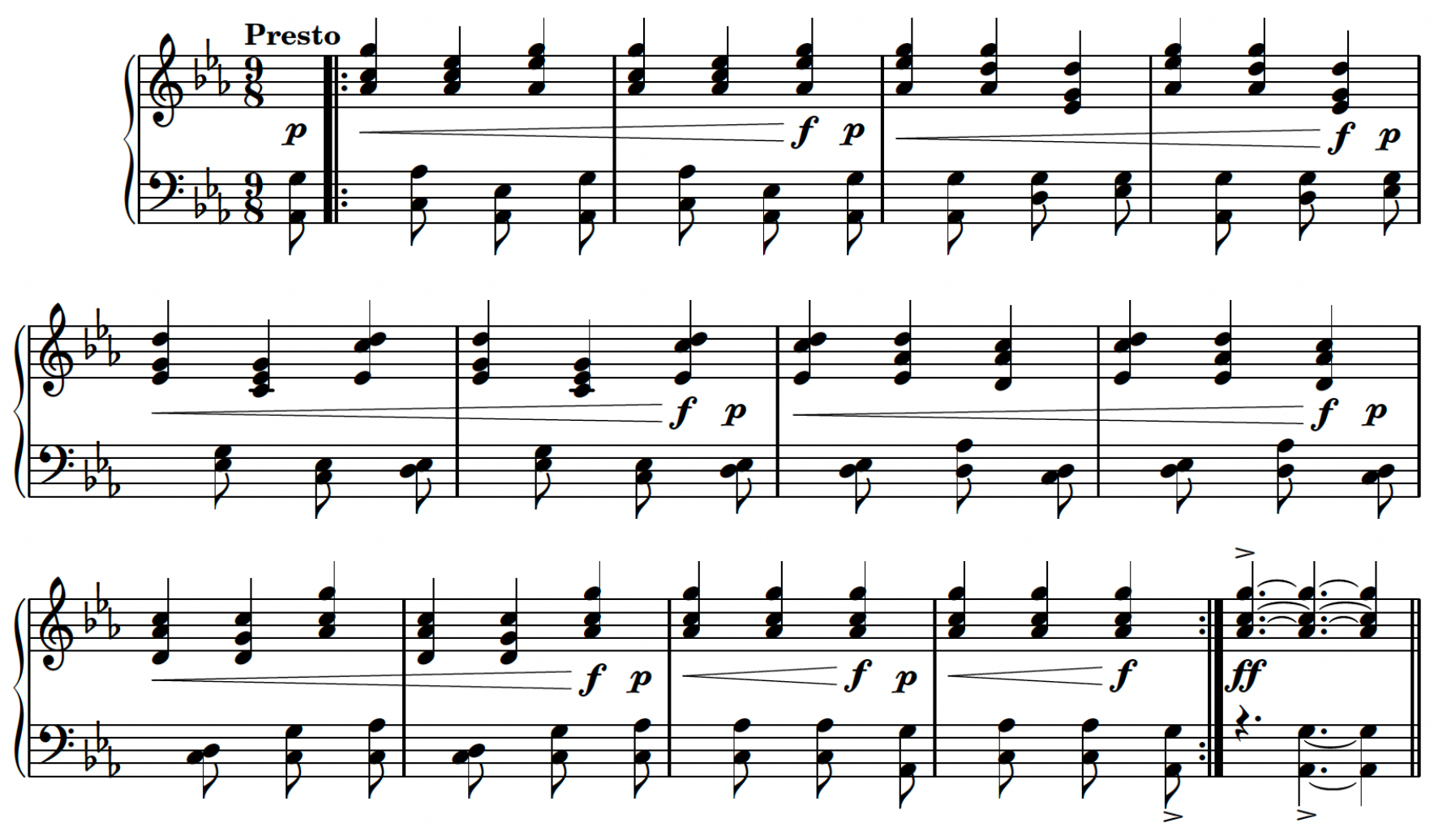}
\caption{A simple composition, ``On the Perimeter,'' based on a Hamiltonian cycle of the pentatonic tonnetz, using 
a Japanese pentatonic scale with the five notes $C$, $D$, $E_{\flat}$, $G$, $A_{\flat}$. The Hamiltonian 0$p$-icosacycle on the left side of Figure \ref{fig:DesarguesLevi} is represented in this piece by the five overlapping 1$p$-hexacycles of the pentatonic tonnetz, each played twice. See Appendix  B for a complete tabulation of the cycles of the Desargues tonnetz.}
\label{Fig:On_the_Perimeter}
\end{figure}

\newpage
\section{Remarks on the Twelve-Tone System}
\label{Remarks on the Twelve-Tone System}

\noindent The twelve-tone system developed by Schoenberg, Webern, Berg and their followers in the Second Viennese School provided an approach to composition that was disciplined and flexible and in the right hands capable of embodying a full range of musical expression. Yet, apart from the general attributes of intelligence and musicality held in common with other such great works of art, the character of the music thus produced is so different from that of the preceding generations that, even if one allows for the extremes of extended tonality, we are tempted to speculate that fundamentally different structural principles are in operation, not merely tweaks of the {\em ancien r\'egime}. Could it be that another tone network is involved? 

The role of unordered hexachords in twelve-tone music has long been recognized. For example,   Lewin (1967) approaches the structure of Schoenberg's Violin Fantasy with a study based on the unordered hexachord $\{F, G, A, B_{\flat}, B, C_{\sharp}\}$ and its complement $\{C, D, D_{\sharp}, E, F_{\sharp}, G_{\sharp}\}$.  
Unordered hexachords also played a prominent role in the work of J. M. Hauer, whose techniques anticipated various aspects of those of Schoenberg. It is not unreasonable to take the view that in any approach to the theory of twelve-tone music the combinatorial geometry of such unordered sextuplets should be investigated in the same spirit and with the same rigour with which we have pursued other tonnetze.

Our musical starting point is a hexachord $H_0 = \{1, 2, 3, 4, 5, 6\}$ -- that is, an unordered set of six distinct tones taken from the chromatic scale with equal temperament, whose elements we have labelled with the numbers one to  six.  These numbers are not, of course, the pitch classes of the six tones; their role is purely combinatorial. For example, in the first hexachord mentioned above, we set $1 = F$, $2 = G$, $3 =  A$, $4 = B_{\flat}$, $5 = B$, $6 =C_{\sharp}$. 
At first glance, such an assignment might seem structureless; but  the permutation group for six objects is  rather special and the highly nontrivial structure of this group is directly related to the particular type of tone network that we wish to propose for twelve-tone music.
The method that we consider is based on Sylvester's  theory of {\em duads} and {\em synthemes} (Sylvester 1844, 1861).  
A duad is an unordered pair of distinct tones taken from $H_0$. 
A little thought shows there are fifteen duads. Thus, a duad $D$ is an unordered set of the form 
\begin{eqnarray}
D = \{i,j\} \,\, {\rm with} \,\,  i,j \in H_0, \, i \neq j. 
\end{eqnarray}
By a syntheme we mean an unordered triple of duads taken in such a way that each element of $H_0$ appears exactly once as one of the elements of the three duads. Hence, a syntheme (Sylvester's term, from the Greek \begin{greek}c<un t'ijhmi\end{greek}) is a set of the form
\begin{eqnarray}
S = \{\{i,j\}, \{k,l\}, \{m,n\} \} \,\, {\rm with} \,\, i,j, k, l, m, n \in H_0, {\rm\, pairwise\,\, distinct}. 
\end{eqnarray}

There are fifteen synthemes altogether. Now, each syntheme contains three duads. But it is also the case that each duad is contained in three distinct synthemes. Thus $\{i,j\}$ is contained in  
$\{\{i,j\}, \{k,l\}, \{m,n\}\}$, $\{\{i,j\}, \{k,m\}, \{n,l\}\}$ and $\{\{i,j\}, \{k,n\}, \{l,m\}\},$ where $i, j, k, l, m, n$ are pairwise distinct. 
 
Hence, we have constructed a combinatorial configuration. More precisely, we have two sets, viz., the set of duads and the set of synthemes. Each has fifteen elements and they are in three-to-three correspondence. We say that the duad $D$ is incident with the syntheme $S$ if it is among the three duads contained in $S$. Inspection of Tables II and IV shows that two synthemes have at most one duad in common. 
This gives us an incidence relation and we can construct the corresponding Levi graph. 

\begin{figure}[htbp]
\centering
\includegraphics[clip,scale=0.83]{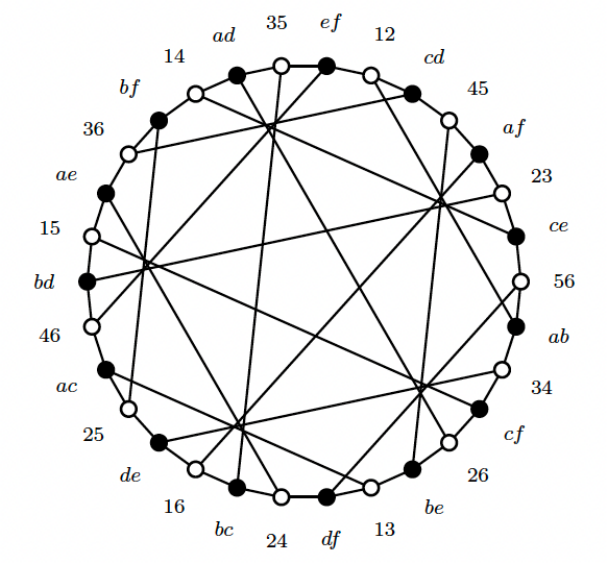}
\caption{The Levi graph of the Cremona-Richmond configuration $\{15_3\}$ maps to a tonnetz for the system of twelve-tone music based on an unordered hexachord. Each white vertex represents a duad (number pair) and each black vertex represents a syntheme (letter pair). Each duad belongs to three synthemes and each syntheme contains three duads.}
\label{fig:Levi graph of configuration 15_3}
\end{figure}

There is one more layer of structure, which comes as a surprise, and this is the following. Consider a collection of five synthemes with the property that each duad appears exactly once. This makes sense since there are fifteen duads and a collection of five synthemes must include fifteen duads. A collection of five ``non-overlapping''  synthemes is called a  ``total''  (again, Sylvester's term). 
There are six such totals that can be constructed. 
Let us label the six totals with lower case Roman letters. Thus, the totals are elements of the set $\{a, b, c, d, e, f\}$. Since we have a second set of six elements, we can form the corresponding ``duads'' and ``synthemes'' for this set as well. 
We shall call these letter duads and letter synthemes, respectively. 

Thus, a letter syntheme is an unordered triple of letter duads. Then by a  letter total we mean a collection of five non-overlapping letter synthemes. 
Two distinct totals have one syntheme in common. Hence, the letter duads of the totals are the original synthemes; and it follows that the letter synthemes are the original duads. 
The dual relations between duads, synthemes, totals, letter duads, letter synthemes, and letter totals can be seen in Tables I to IV in Appendix A (Richmond 1900, Baker 1925, Coxeter 1950, 1958). 


\begin{figure}[htbp] 
\centering
\includegraphics[clip,scale=0.83]{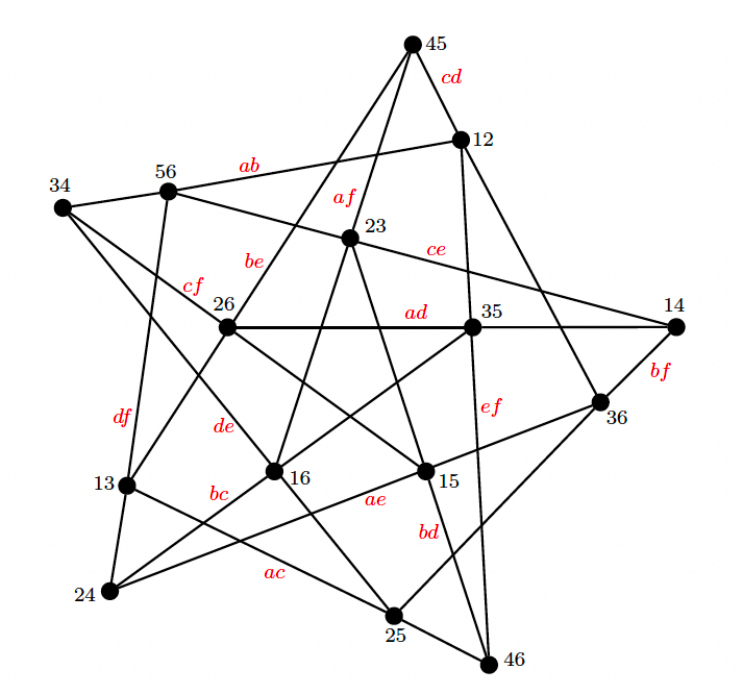}		
\caption{The tonnetz corresponding to an unordered hexachord maps to a configuration $\{15_3\}$ in $\mathbb {R}^2$ consisting of fifteen points and fifteen lines having the property that three points lie on each line and three lines pass through each point. The points correspond to duads of tones and the lines correspond to unordered triples of duads.} 
\label{configuration 15_3}
\end{figure} 


The advantage of using the number duads and the letter duads 
is that this allows us to label the resulting Levi graph neatly, where the number duads, which correspond to letter synthemes, are represented by white vertices and the letter duads, which correspond to number synthemes, are represented by black vertices. 
One sees that there is indeed a rich hierarchical structure implicit in the specification of a hexachord.  
The Levi graph of the Cremona-Richmond $\{15_3\}$ is shown in Figure \ref{fig:Levi graph of configuration 15_3} (Tutte 1947, Coxeter 1950, Pisanski and Servatius 2013, Brier and Bryant 2021). We can take this graph as representing a tonnetz for the twelve-tone system when an unordered hexachord has been selected as a primitive. Its girth is eight: there are no hexacycles.

A construction of the  Cremona-Richmond $\{15_3\}$ is as follows. We find it convenient here to use the language of projective geometry. We begin with six points $p_1, .\,.\,., p_6$ in general position in  $\mathbb{RP}^4$. When joined pairwise, these points determine fifteen  lines. 
Let us write $L_{12} = p_1\vee p_2$ for the line joining $p_1$, $p_2$, and $H_{3456} = \vee (p_3 p_4 p_5 p_6)$ for  the complementary hyperplane joining $p_3$, $p_4$, $p_5$, $p_6$. This hyperplane is an $\mathbb{RP}^3$. 
Each of the fifteen lines meets its complementary hyperplane at a point. 
Thus,  we can write $P_{12} = L_{12} \wedge H_{3456}$ for the meet of $L_{12}$ and  $H_{3456}$. 

There are evidently fifteen such ``meeting points.'' 
On the other hand, for example, we also have 
$P_{34} = L_{34} \wedge H_{1256}$ and $P_{56} = L_{56} \wedge H_{1234}$. 
Let us write $S_{12,34,56}$ for the line given by the meet of the three hyperplanes. Since $H_{1234}$ and $H_{1256}$ both contain $L_{12}$ it follows that the plane of intersection $D_{12}$ of these two hyperplanes contains $L_{12}$. 

Then  since $D_{12}$ meets $ H_{3456}$ at $S_{12,34,56}$ it must be that $L_{12}$ meets $S_{12,34,56}$ at  $P_{12}$. By symmetry, all three of the ``duad'' meeting points $P_{12}$, $P_{34}$, $P_{56}$ lie on the ``syntheme'' line  $S_{12,34,56}$.  
The fifteen meeting points are the duads of the configuration; and the fifteen lines of three-at-a-time hyperplane intersection, each of which contains three of the duads, are the synthemes of the configuration. 
That gives the four-dimensional version of the construction, which can then be projected to a plane to provide a $\{15_3\}$ in two dimensions,  the Cremona-Richmond configuration. 

Once the result has been obtained on the real projective plane, it can be Euclideanized and that leads us to Figure \ref{configuration 15_3}. For further details see Richmond (1900),  Baker (1925), and Coxeter (1950). 

\vspace{0.20cm}
\begin{Proposition}
A tonnetz can be associated with the selection of an unordered hexachord in twelve-tone music, given by a self-dual $\{15_3\}$ in $\mathbb {R}^{2}$ known as the Cremona-Richmond configuration.  In this tonnetz, the $15$ ``major'' duad chords are represented by points and the $15$  ``minor'' syntheme chords are represented by lines. The $45$ incidence relations between the $15$ points and the $15$ lines determine the edges of the corresponding Levi graph.
\label{prop:15_3 proposition}
\end{Proposition}
\vspace{0.25cm}
 
 \begin{figure}[htbp]
 \centering
\includegraphics[scale=0.55]{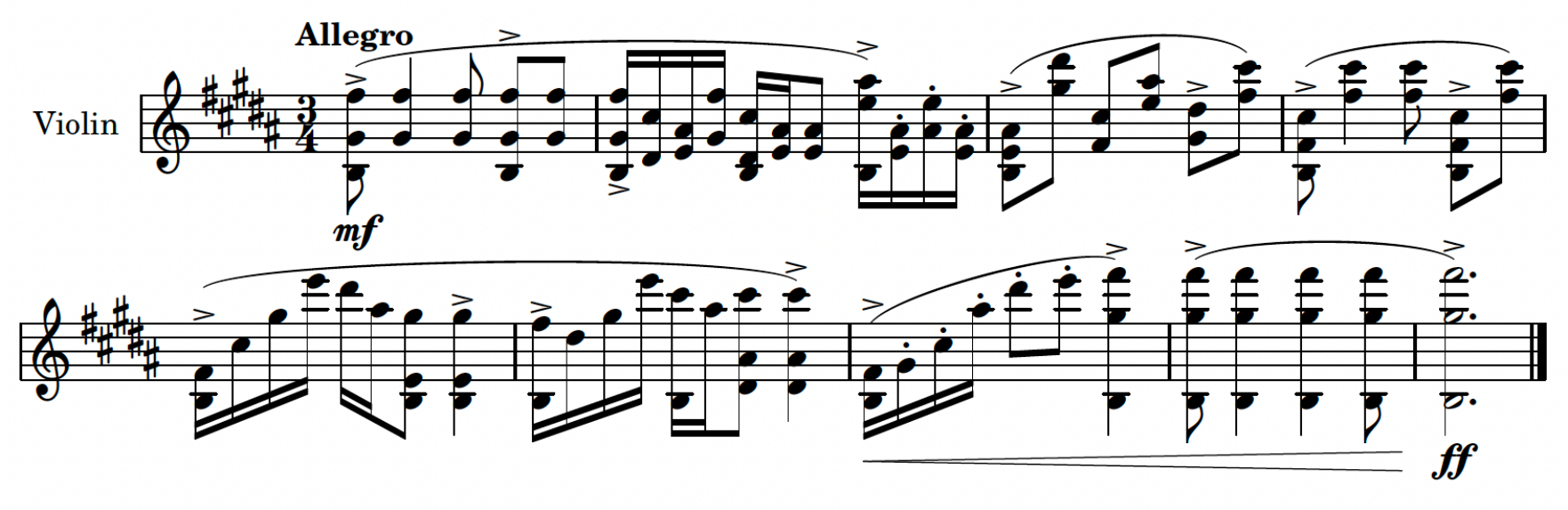}
\caption{``Decacycle for Synthetic Violin,'' based on a 5$p$-decacycle of the 12-tone tonnetz constructed from a hexachord of the form $[1, 2, 3, 4, 5, 6] = [F_{\sharp}, G_{\sharp}, C_{\sharp}, D_{\sharp}, E, A_{\sharp}]$. The duads and synthemes of the decacycle are given by $\langle 12, ab, 56, df, 13, ac, 25, bf, 36, cd, 12 \rangle$, forming the sides and vertices  of a pentagon in the Cremona-Richmond configuration of Figure  \ref{configuration 15_3}.   The structure of the pentagon can also be seen in Figure \ref{fig:Levi graph of configuration 15_3}. The note $B$ is used in this piece to ground the music. The relations between the letter duads and the number synthemes are shown in Table II of Appendix A.}
\label{Fig:Decacycle for Violin}
\end{figure}
%


There are 144 Hamiltonian cycles altogether. There are no hexacycles, but there are 90 octacycles. Among these are the five overlapping 1$p$-octacycles, which range about the perimeter Hamiltonian cycle in a manner similar to the arrangement of the five overlapping 1$p$-hexacycles in the pentatonic tonnetz.
One might not have thought there were affinities between the pentatonic system and the twelve-tone system, but there are. 

In fact, there is a curious kind of ``five-ness'' that pervades the twelve-tone system that is clearly evident in our approach, though not perhaps so obvious in traditional treatments of serialism. If any one tone of the hexachord is singled out (in this case, the tone 4), the remaining tones can be arranged with pentagonal symmetry. An example of ``five-ness'' in the traditional approach to the twelve-tone system is the involutory map from the set of integers mod 12 to itself given by multiplication by 5. 
In Figure \ref{Fig:Decacycle for Violin} we present a short composition based on this pentagonal symmetry.

\section{Subset relations vs voice-leading relations}
\label{sec:Subset relations vs voice-leading relations}

\noindent One criticism of the approach we have taken to the construction of tonnetze for the pentatonic and twelve-tone systems, and the more general systems  considered in Proposition \ref{Generalization proposition},  might be that it is based on subset inclusion relations rather than the voice-leading relations that form the basis of our treatments of the Eulerian tonnetz and the Tristan-genus tonnetz in B-H (2026). 
Nonetheless, the Levi graphs of the Fano configuration, the Eulerian D222 configuration, the Tristan-genus D228 configuration, the pentatonic Desargues configuration, and the twelve-tone Cremona-Richmond configuration do bear a striking family resemblance, hinting at some commonality in the superficially divergent methodologies. 
This suggests that we should look at subset relations as a possible basis for the construction of the Eulerian tonnetz. 

Consider the twelve tones of the chromatic scale and the twelve major triads. Clearly, each triad contains three tones; but it is also the case that each tone is contained in exactly three major triads. For example, $C$ is contained in $CM$, $FM$ and $A_{\flat}M$. So, we have a three-to-three map between the tone set and the major-triad set. It is true as well that any tone is contained in exactly three minor triads -- so one obtains another three-to-three map. What do these maps look like? What is their relation to the tonnetz? We can construct the relevant Levi graphs, and the results are shown in Figure \ref{note-chord tonnetz}. We obtain the following.

\begin{Proposition}
The incidence structure of the system consisting of the major chords and the pitch classes of the chromatic scale is a $\{12_3\}$ of Daublebsky von Sterneck type ${\rm D}222$.   Likewise, the incidence structure of the system of the minor chords and the pitch classes of the chromatic scale is a $\{12_3\}$ of Daublebsky von Sterneck type ${\rm D}222$.
\label{prop: incidence structures of note-major and note-minor}
\end{Proposition}

So, evidently, the combinatorial structure of the tonnetz is already implicit in the relation between the triads of either mode and the tones within these triads.

\begin{figure}[!htbp] 
\centering
\includegraphics[clip,scale=0.60]{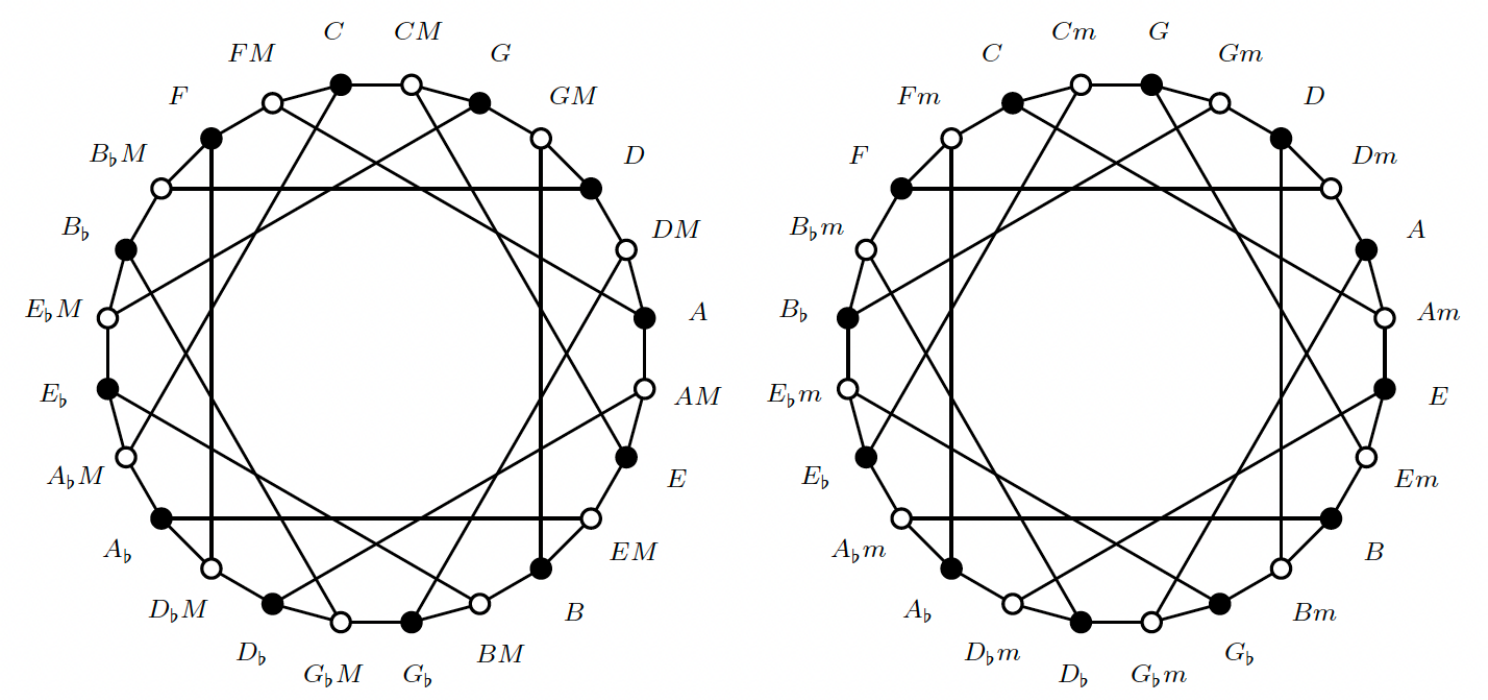}
\caption{Levi graphs for the pitch-class-to-major-triad and pitch-class-to-minor-triad tonnetze. On the left, each major triad (white vertex) is joined to three pitches (black vertices) and each pitch is joined to three major triads with the obvious set-inclusion relations. The resulting graph has the same incidence structure as that of the Eulerian tonnetz. An analogous graph can be constructed for the minor triads, shown on the right, again with the same incidence structure.}
\label{note-chord tonnetz}
\end{figure}

\begin{Proposition}
The Eulerian tonnetz can be constructed on the basis of set-inclusion relations, without reference to voice leading, distance, work, efficiency or parsimony.  
\label{prop: set inclusion relations}
\end{Proposition}

\proof We begin with the twelve tones of the chromatic scale and the twelve major triads. Associated with each tone are the three major triads to which it belongs. For example, $C$ belongs to $CM$, $A_{\flat}M$ and $FM$. Each triad contains three tones, so we get a three-to-three map. 
Then we construct a tessellation of the plane by hexagons in such a way that the tones and triads sit on the vertices in an alternating pattern. Each tone is connected by edges to the three associated major triads, and each triad is connected by edges to the three tones of which it is composed. Each hexagon in the tessellation then consists of a cycle of tones and major triads. 
For example, $C$ is contained in the cycle $\langle C, A_{\flat}M, E_{\flat}, E_{\flat}M, G, 
CM, C \rangle$. Now, the three tones in any given cycle can be combined to form a minor triad, so we mark the centre of that hexacycle with that minor triad. 
For instance, the cycle just considered is marked with $Cm$ at the centre. Each minor triad is connected to the three tones it contains and to the three major triads in the hexagon at the centre of which the minor triad sits. The resulting triangular tessellation of the plane is the complete triadic tonnetz of Figure \ref{fig:tripartite Levi graph}. If we rub out the tones and the links between tones and triads, we obtain the Eulerian tonnetz of Figure 2 in B-H (2026). \endproof

\begin{figure}[!htbp] 
\centering
\includegraphics[clip,scale=0.55]{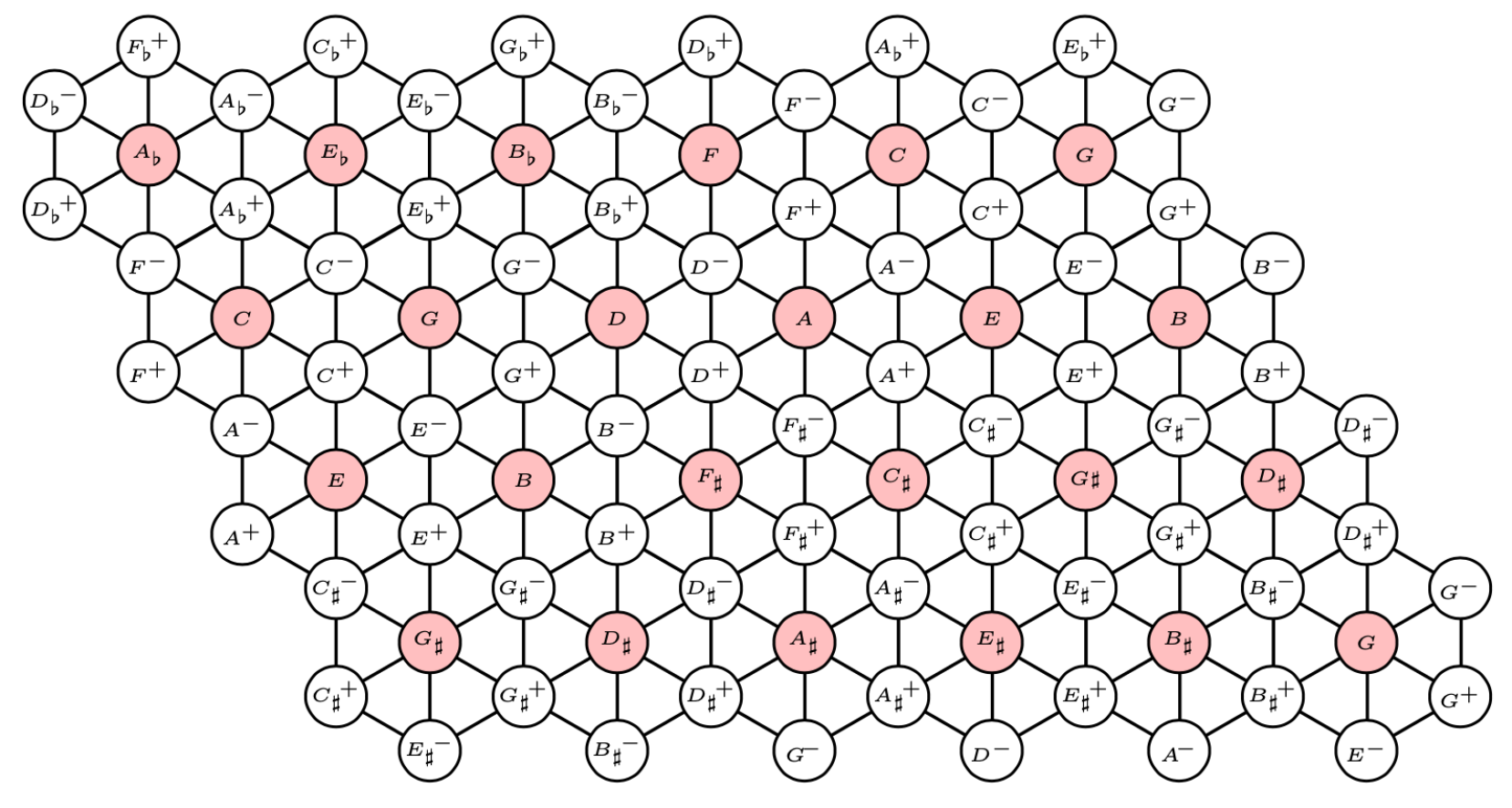}
\caption{The complete triadic tonnetz is a face-centred hexagonal tessellation admitting an infinite tripartite graph based on an underlying triangular tessellation. This array of hexagons, each with a pitch at its centre, is such that each triad is surrounded by its constituent pitches. If the pitches and the  edges connected to them are rubbed out, one is left with the tessellation of the plane by major and minor triads shown in Figure 2 of B-H (2026).  If one or the other of the remaining classes of edges is rubbed out instead, we obtain the Levi graphs of Figure \ref{note-chord tonnetz} of the present paper. Here, the major and minor triads are marked by plus and minus signs. }
\label{fig:tripartite Levi graph}
\end{figure}

The idea that the Eulerian tonnetz can be constructed on the basis of set-inclusion relations comes as a surprise, but the matter can be readily understood. As pointed out in B-H (2026), the tonnetz admits two different representations as a tessellation, one in which the pitches are at the vertices and one in which the major and minor triads are at the vertices. 
The pitch-based tonnetz is essentially that of Euler and has been used by Cohn (2012) and others as a starting point for music analysis. Triadic versions of the tonnetz appear in Waller (1978), Douthett and Steinbach (1998), Tymoczko (2011) and elsewhere.  
We say that a major triad and a minor triad are incident if they have two tones in common. This relation also functions as a basis for voice leading. The point is that one need not start from that relation as primary. 
If one takes the pitch-class-to-major-triad tonnetz of Figure \ref{note-chord tonnetz} as given, we observe that the minor triads are in one-to-one correspondence with the 2$p$-hexacycles of this graph. The three minor triads to which a tone belongs are the three 2$p$-hexacycles containing that tone. 
For example, the tone $C$ is contained in the 2$p$-hexacycles corresponding to $Cm$, $Am$ and $Fm$. This point of view is interesting since it breaks the duality between the major and minor triads and to some extent restores the primacy of the major triads. 
According to this view, the minor triads can be identified as secondary structures -- namely, the 2$p$-hexacycles. One can also ask about the role of the four 3$p$-hexacycles of the pitch-class-to-major-triad tonnetz. Inspection of the Levi graph shows that these correspond  to the four augmented triads.
In Figure \ref{fig:tripartite Levi graph}, the pitch set and the two triad sets are combined in the form of a face-centred hexagonal tessellation that treats these sets (the set of pitches, the set of major triads, and the set of minor triads) on an equal footing. A similar graph appears in Cubarsi (2024). 
There are three types of vertices (pitch, major triad, minor triad) and three types of edges (pitch to major triad, pitch to minor triad, major triad to minor triad). Each pitch vertex is joined by three pitch-class-to-major-triad edges and three pitch-class-to-minor-triad edges. Each major-triad vertex is joined by three major-triad-to-pitch-class edges and three major-triad-to-minor-triad edges. Each minor-triad vertex is joined by three minor-triad-to-pitch-class edges and three minor-triad-to-major-triad edges. 
In this way, we obtain the tripartite graph of a combinatorial configuration of three element types that can be represented by a symbol of the form $\{12_{3,3}, 12_{3,3}, 12_{3,3} \}$. 

\section{Summary and Conclusions}

\begin{quote} . \,. \,. there was a time when the study of configurations was considered the most important branch of all geometry. 

\hfill ---D.~Hilbert \& S.~Cohn-Vossen,~{\em Geometry and the Imagination} (1952)
\end{quote}

\noindent Now we are in a position to take stock of our observations concerning the various and intertwined relations between configurations, tessellations, bipartite graphs and tone networks, with attention both to what we have learned in the present article as well as in B-H (2026). A common emergent structure characteristic of a number of different music systems is that of a biregular graph. The system of triads associated with a diatonic key of a given mode gives rise to a biregular graph of type  $\{7_3\}$ and girth four. This graph, resembling a steering wheel, has seven 1$p$-tetracyclic ``cushions" along the rim of the Hamiltonian cycle. The system of diatonic tetrads (seventh chords) again results in a biregular graph of type  $\{7_3\}$ but of a distinct variety, of girth six, taking the form of the Levi graph of the Fano configuration. 

The seven seventh chords associated with a fixed diatonic scale (major or minor) have the property that any one of them overlaps on at least one tone (but possibly two or three) with all six of the others. 
For example, in the key of $D$ minor, the $Dm^7$ chord admits an overlap with $C\rm{V}^7$ on the note $C$, with whole-step transitions from $D$ to $E$ and from $F$ to $G$, and a semitone transition from $A$ to $B_\flat$. The diatonic seventh-chord tonnetz $\{7_3\}$ captures all these elements, which can be represented equivalently as a biregular graph $\{7_3\}$, as a configuration $\{7_3\}$, and as an associated tessellation, each offering different perspectives on the same architecture -- for it is of the nature of mathematics that when a multiplicity of views prevail, then no one line of sight is the end of the discussion; rather, the underlying object is the totality of its representations, and more. 

The pentatonic system may seem meagre when it is contrasted with the rich environment of the heptatonic tonal system of the common practice period, but other structures come into play when one considers tone clusters with two or three elements in each cluster and the result is not without interest. We have seen that the two aggregates of cluster form a biregular graph $\{10_3\}$ with an associated configuration, that of Desargues. Can something similar be achieved with the three and four note clusters inherent within the diatonic system? Our Proposition \ref{Generalization proposition} shows that the answer is yes, that the system of three-note and four-note clusters within the diatonic scale forms a biregular graph $\{35_4\}$. For example, $\{1, 2, 3\}$ maps to $\{1, 2, 3, 4\}$, $\{1, 2, 3, 5\}$, $\{1, 2, 3, 6\}$, $\{1, 2, 3, 7\}$, whereas $\{1, 2, 3, 4\}$ contains $\{1, 2, 3\}$, $\{1, 2, 4\}$, $\{1, 3, 4\}$, $\{2, 3, 4\}$. The system of diatonic triads and diatonic tetrads for any given key forms a $\{7_2\}$ subconfiguration of this $\{35_4\}$, in the form of a heptagon. For example, in the key of $C$ major the resulting tessareskaidekacycle takes the form $\langle CM, CM^7, Em, Em^7, GM, G{\rm V}^7, B\degree, B\varnothing^7, Dm, Dm^7, FM, FM^7, Am, Am^7, CM\rangle$. 

We have had occasion to consider three distinct biregular graphs of type $\{12_3\}$. The first is the Euler-Oettingen-Riemann tonnetz -- which we have labelled the ``Eulerian'' tonnetz. This $\{12_3\}$ in its guise as a biregular graph is the Levi graph of the configuration D222 described by Daublebsky von Sterneck (1895), as we pointed out in B-H (2026). This configuration captures all aspects of the Eulerian tonnetz and gives it a geometric representation as a system of points and lines in Euclidean space, the alignments and intersections of which carry a rich musical significance. 
The Eulerian tonnetz characterizes the well-known Oettingen-Riemann relations between major and minor triads, but can also be used to model other sets of relations between chord blocks, as we have shown. 
 A second tonnetz of type $\{12_3\}$ can be modelled on Daublebsky von Sterneck's D228 to account for the relations between dominant sevenths and half-diminished sevenths typical of Wagnerian harmonies. The D222 and the D228 form part of a series of five connected biregular graphs of type $\{12_3\}$ upon which the cyclic group $Z_{12}$ acts naturally in a transitive way. The others include Daublebsky von Sterneck's D226, and a pair of graphs of girth four, one of which is a steering wheel marked by twelve 1$p$-tetracycles  (cushions) along the rim of the Hamiltonian trajectory, the other of which being a ``bicycle wheel" marked by six 1$p$-dodecacycles, or equivalently, six 2$p$-tetracycles. But there are also some tonnetze of type $\{12_3\}$ that split into disconnected components.  First, there is the Archimedean tonnetz of B-H (2026), so-called  because its two distinct components each admit representation by an Archimedean tessellation of the plane involving squares, hexagons and dodecagons, known to Kepler. The Archimedean tonnetz can be thought of as a pair of bicycle wheels, each with six rim-to-rim spokes.  An analogous tonnetz of type $\{12_3\}$ takes the form of a pair of half-sized steering wheels, each with six cushions. Another $\{12_3\}$ splits into three octagonal bicycle wheels, each with four rim-to-rim spokes (alternatively, they can be viewed as a trio of steering wheels with four cushions each). Then there is a four-component $\{12_3\}$, each part of which consists of a hexagon with opposite vertices joined, giving a kind of hexagonal bicycle wheel with three rim-to-rim spokes. These tonnetze are ripe with musical potentiality.
 Finally, we have considered here a $\{15_3\}$ that shows promise in the context of the twelve-tone system by means of the combinatorial structures implicit in Sylvester's duads and synthemes. 

All of these biregular graph structures have elements in common and as such represent arrays of interrelated musical ``resources'' that both the composer and the analyst can draw upon in their work. Each such resource is different and has its own distinctive features. This pattern is typical of the type of organization one often encounters in mathematics, where diverse mathematical objects arise, all related, each nonetheless with its own characteristic features -- as one finds in algebraic geometry. Certainly, patterns appear and reappear, but in such a way that there is no uniform structure to the arrangement as a whole. 

On the contrary, variegation and elaboration seem to be the rule, and so it is too with the layers of new musical structures arising from the set-theoretic inclusion relations typical of the tonnetze that we have considered. We are led to a view of ``abstract musical resources'' as the totality of the orderly relations that can be established between biregular graphs (or multiregular graphs) and tone sets. These and other mathematical consequences of the present investigation we hope to pursue at greater length elsewhere.

\vspace{1.00cm}
\noindent {\bf References}

\newpage
\section*{APPENDIX A: Tables of Duads and Synthemes}
\begin{table}[!htbp]
\scriptsize
\centering
\begin{tabular}{ |p{0.3cm}|p{1.4cm}|p{1.4cm}|p{1.4cm}|p{1.4cm}|p{1.4cm}|}
\hline
 \multicolumn{6}{|c|}{\textbf{Letters as totals of number synthemes}} \\
\hline
$a$ & 12, 34, 56 & 13, 25, 46 & 14, 26, 35 & 15, 24, 36 & 16, 23, 45 \\
$b$ & 12, 34, 56 & 16, 24, 35 & 15, 23, 46 & 13, 26, 45 & 14, 25, 36 \\
$c$ & 13, 25, 46 & 16, 24, 35 & 12, 36, 45 & 14, 23, 56 & 15, 26, 34 \\
$d$ & 14, 26, 35 & 15, 23, 46 & 12, 36, 45 & 16, 25, 34 & 13, 24, 56 \\
$e$ & 15, 24, 36 & 13, 26, 45 & 14, 23, 56 & 16, 25, 34 & 12, 35, 46 \\
$f$ & 16, 23, 45 & 14, 25, 36 & 15, 26, 34 & 13, 24, 56 & 12, 35, 46 \\
\hline
\end{tabular}
\captionof{table}{Each of the six letters corresponds to a collection of five number synthemes.}
\end{table}

\begin{table}[!htbp]
\scriptsize
\centering
\begin{tabular}{ |p{0.5cm}|p{1.8cm}||p{0.5cm}|p{1.8cm}||p{0.5cm}|p{1.8cm}|}
\hline
 \multicolumn{6}{|c|}{\textbf{Duads of letters as synthemes of numbers}} \\
\hline
$ab$ & 12, 34, 56 & $bc$ & 16, 24, 35 & $ce$ & 14, 23, 56\\
 $ac$ & 13, 25, 46  & $bd$ &15, 23, 46 & $cf$ &15, 26, 34\\
 $ad$ &14, 26, 35 & $be$ &13, 26, 45 & $de$ &16, 25, 34 \\
 $ae$ &15, 24, 36 & $bf$ &14, 25, 36 & $df$ &13, 24, 56\\
 $af$ & 16, 23, 45 & $cd$ &12, 36, 45 & $ef$ &12, 35, 46\\
\hline
\end{tabular}
\captionof{table}{Each duad of letters corresponds to a syntheme of numbers.}
\end{table} 

\begin{table}[!htbp]
\scriptsize
\centering
\begin{tabular}{ |p{0.3cm}|p{1.4cm}|p{1.4cm}|p{1.4cm}|p{1.4cm}|p{1.4cm}|}
\hline
 \multicolumn{6}{|c|}{\textbf{Numbers as totals of letter synthemes}} \\
\hline
$1$ & $ab, cd, ef$ & $ac, be, df$ & $ad, bf, ce$ & $ae, bd, cf$ & $af, bc, de$ \\
$2$ & $ab, cd, ef$ & $af, bd, ce$ & $ae, bc, df$ & $ac, bf, de$ & $ad, be, cf$ \\
$3$ & $ab, cf, de$ & $ad, bc, ef$ & $ae, bf, cd$ & $ac, be, df$ & $af, bd, ce$ \\
$4$ & $af, be, cd$ & $ac, bd, ef$ & $ad, bf, ce$ & $ae, bc, df$ & $ab, cf, de$ \\
$5$ & $ab, ce, df$ & $ae, bd, cf$ & $ac, bf, de$ & $ad, bc, ef$ & $af, be, cd$ \\
$6$ & $af, bc, de$ & $ad, be, cf$ & $ae, bf, cd$ & $ac, bd, ef$ & $ab, ce, df$ \\
\hline
\end{tabular}
\captionof{table}{Each of the numbers corresponds to a collection of five letter synthemes.}
\end{table}

\begin{table}[!htbp]
\scriptsize
\centering
\begin{tabular}{ |p{0.5cm}|p{1.8cm}||p{0.5cm}|p{1.8cm}||p{0.5cm}|p{1.8cm}|}
\hline
 \multicolumn{6}{|c|}{\textbf{Duads of numbers as synthemes of letters}} \\
\hline
 $12$ & $ab, cd, ef$ & $23$ & $af, bd, ce$ & $35$ & $ad, bc, ef$\\
 $13$ & $ac, be, df$  & $24$ & $ae, bc, df$ & $36$ & $ae, bf, cd$\\
 $14$ & $ad, bf, ce$ & $25$ & $ac, bf, de$ & $45$ & $af, be, cd$\\
 $15$ & $ae, bd, cf$ & $26$ & $ad, be, cf$ & $46$ & $ac, bd, ef$\\
 $16$ & $af, bc, de$ & $34$ & $ab, cf, de$ & $56$ & $ab, ce, df$\\
 \hline
\end{tabular}
\captionof{table}{Each duad of numbers corresponds to a syntheme of letters.}
\end{table}

\newpage

\newpage 
\FloatBarrier
\appendix
\section*{APPENDIX B: Cycle Counts}

\begin{table}[!htbp]
\centering
\scriptsize

\begin{minipage}[t]{0.40\textwidth}
\vspace{22pt}
\centering
{\setlength{\tabcolsep}{3pt}
\begin{tabular}{|c|cccccccc|c|}
\hline
\multicolumn{10}{|c|}{\textbf{Cycle Count for Diatonic Seventh Tonnetz }} \\
\hline
$ $ & $0$ & $1$ & $2$ & $3$ & $4$ & $5$ & $6$ & $7$ & $\text{Total}$ \\
\hline
$6$  & $0$   & $7$   & $14$   & $7$   & $0$   & $0$   & $0$   & $0$   & $28$ \\
$8$  & $0$   & $0$   & $7$    & $14$  & $0$   & $0$   & $0$   & $0$   & $21$ \\
$10$ & $0$   & $7$   & $7$    & $28$  & $35$  & $7$   & $0$   & $0$   & $84$ \\
$12$ & $0$   & $0$   & $7$    & $7$   & $21$  & $21$  & $0$   & $0$   & $56$ \\
$14$ & $1$   & $0$   & $0$    & $7$   & $0$   & $7$   & $7$   & $2$   & $24$ \\
\hline
$\text{Total}$ & $1$ & $14$ & $35$ & $63$ & $56$ & $35$ & $7$ & $2$ & $213$ \\
\hline
\end{tabular}}
\end{minipage}
\hfill
\begin{minipage}[t]{0.57\textwidth}
\vspace{0pt}
\centering
{\setlength{\tabcolsep}{3pt}
\begin{tabular}{|c|ccccccccccc|c|}
\hline
\multicolumn{13}{|c|}{\textbf{Cycle Count for Pentatonic Tonnetz}} \\
\hline
$ $ & $0$ & $1$ & $2$ & $3$ & $4$ & $5$ & $6$ & $7$ & $8$ & $9$ & $10$ & $\text{Total}$ \\
\hline
$6$  & $0$ & $5$ & $10$ & $5$  & $0$   & $0$   & $0$   & $0$   & $0$   & $0$   & $0$   & $20$ \\
$8$  & $0$ & $0$ & $10$ & $20$ & $0$   & $0$   & $0$   & $0$   & $0$   & $0$   & $0$   & $30$ \\
$10$ & $0$ & $5$ & $20$ & $45$ & $50$  & $12$  & $0$   & $0$   & $0$   & $0$   & $0$   & $132$ \\
$12$ & $0$ & $5$ & $5$  & $30$ & $60$  & $45$  & $5$   & $0$   & $0$   & $0$   & $0$   & $150$ \\
$14$ & $0$ & $0$ & $20$ & $40$ & $120$ & $125$ & $110$ & $5$   & $0$   & $0$   & $0$   & $420$ \\
$16$ & $0$ & $5$ & $0$  & $25$ & $60$  & $105$ & $80$  & $85$  & $0$   & $0$   & $0$   & $360$ \\
$18$ & $0$ & $0$ & $5$  & $20$ & $30$  & $40$  & $95$  & $90$  & $30$  & $10$  & $0$   & $320$ \\
$20$ & $1$ & $0$ & $0$  & $0$  & $0$   & $7$   & $5$   & $0$   & $5$   & $5$   & $1$   & $24$ \\
\hline
$\text{Total}$ & $1$ & $20$ & $70$ & $185$ & $320$ & $334$ & $295$ & $180$ & $35$ & $15$ & $1$ & $1456$ \\
\hline
\end{tabular}}
\end{minipage}
\vspace{0.25cm}
\caption{Cycle count for the Fano Levi graph of the diatonic seventh tonnetz on the left and for the Desargues Levi graph of the pentatonic tonnetz on the right. The rows are labelled by cycle length and the columns are labelled by the number of $p$-crossings. Each entry shows the number of cycles of a given length with the given number of crossings. Totals are given for each row and each column.}
\end{table}
\vspace{0.25cm}

\begin{table}[!htbp]
\scriptsize
\centering
{\setlength{\tabcolsep}{3pt}
\begin{tabular}{|c|ccccccccccccccc|c|}
\hline
\multicolumn{17}{|c|}{\textbf{Cycle Count for Twelve-Tone Tonnetz}} \\
\hline
$ $ & $0$ & $1$ & $2$ & $3$ & $4$ & $5$ & $6$ & $7$ & $8$ & $9$ & $10$ & $11$ & $12$ & $13$ & $14$ & $\text{Total}$ \\
\hline
$8$  & $0$ & $5$ & $30$ & $45$ & $10$ & $0$ & $0$ & $0$ & $0$ & $0$ & $0$ & $0$ & $0$ & $0$ & $0$ & $90$ \\
$10$ & $0$ & $5$ & $5$  & $30$ & $25$ & $7$ & $0$ & $0$ & $0$ & $0$ & $0$ & $0$ & $0$ & $0$ & $0$ & $72$ \\
$12$ & $0$ & $0$ & $10$ & $70$ & $135$ & $80$ & $5$ & $0$ & $0$ & $0$ & $0$ & $0$ & $0$ & $0$ & $0$ & $300$ \\
$14$ & $0$ & $5$ & $40$ & $105$ & $295$ & $365$ & $255$ & $15$ & $0$ & $0$ & $0$ & $0$ & $0$ & $0$ & $0$ & $1080$ \\
$16$ & $0$ & $0$ & $15$ & $85$ & $280$ & $465$ & $530$ & $230$ & $15$ & $0$ & $0$ & $0$ & $0$ & $0$ & $0$ & $1620$ \\
$18$ & $0$ & $5$ & $5$ & $80$ & $350$ & $845$ & $1085$ & $1130$ & $320$ & $20$ & $0$ & $0$ & $0$ & $0$ & $0$ & $3840$ \\
$20$ & $0$ & $0$ & $35$ & $135$ & $325$ & $1007$ & $1685$ & $2080$ & $1655$ & $580$ & $22$ & $0$ & $0$ & $0$ & $0$ & $7524$ \\
$22$ & $0$ & $5$ & $0$ & $55$ & $245$ & $575$ & $1415$ & $2035$ & $2070$ & $1410$ & $460$ & $10$ & $0$ & $0$ & $0$ & $8280$ \\
$24$ & $0$ & $5$ & $0$ & $45$ & $155$ & $535$ & $855$ & $1685$ & $2345$ & $2135$ & $1335$ & $345$ & $10$ & $0$ & $0$ & $9450$ \\
$26$ & $0$ & $0$ & $25$ & $25$ & $85$ & $245$ & $615$ & $895$ & $1390$ & $1610$ & $1410$ & $965$ & $285$ & $10$ & $0$ & $7560$ \\
$28$ & $0$ & $0$ & $0$ & $25$ & $5$ & $20$ & $105$ & $140$ & $150$ & $255$ & $305$ & $200$ & $170$ & $60$ & $5$ & $1440$ \\
$30$ & $1$ & $0$ & $0$ & $0$ & $0$ & $7$ & $0$ & $15$ & $20$ & $20$ & $16$ & $10$ & $30$ & $20$ & $5$ & $144$ \\
\hline
$\text{Total}$ & $1$ & $30$ & $165$ & $700$ & $1910$ & $4151$ & $6550$ & $8225$ & $7965$ & $6030$ & $3548$ & $1530$ & $495$ & $90$ & $10$ & $41400$ \\
\hline
\end{tabular}}
\vspace{0.25cm}
\caption{Cycle count for the Levi graph of the Cremona-Richmond configuration for the twelve-tone tonnetz.}
\end{table}

\vspace{0.25cm}

\end{document}